\documentclass[10pt]{amsart}

 \usepackage{graphicx}
 \usepackage[all]{xy}
 \usepackage{amscd,amsfonts,amssymb,amsmath,amsthm,latexsym}
\usepackage{hyperref}

\begin{document}

\newcommand{\suchthat}{\ | \ }
\newcommand{\myid}{1\hspace{-0.125cm}1}
\newcommand{\marked}{\mathbb{M}}
\newcommand{\punct}{\mathbb{P}}
\newcommand{\surf}{(\Sigma,\marked,\punct)}
\newcommand{\leafyG}{G^{\text{{\tiny \rotatebox[origin=r]{90}{$\multimap$}}}}}
\newcommand{\ribsurf}[1]{\Sigma(#1)}

\theoremstyle{plain}
    \newtheorem{theorem}{Theorem}[section]
    \newtheorem{lemma}[theorem]{Lemma}
    \newtheorem{proposition}[theorem]{Proposition}

\theoremstyle{definition}
    \newtheorem{defi}[theorem]{Definition}
    \newtheorem{ex}[theorem]{Example}
    \newtheorem{remark}[theorem]{Remark}

\theoremstyle{remark}
    \newtheorem{case}{Case}

\numberwithin{equation}{section}

\title[On the resolution of kinks of curves on punctured surfaces]{On the resolution of kinks of curves\\ on punctured surfaces}
\author{Christof Geiss}
\address{Christof Geiss\newline
Instituto de Matem\'aticas, UNAM, Mexico}
\email{christof.geiss@im.unam.mx}
\author{Daniel Labardini-Fragoso}
\address{Daniel Labardini-Fragoso\newline
Instituto de Matem\'aticas, UNAM, Mexico}
\email{labardini@im.unam.mx}

\subjclass{57K20, 13F60, 18B40}

\begin{abstract}    
Let $\surf$ be a surface with marked points $\marked\subseteq \partial\Sigma\neq\varnothing$ and punctures $\punct\subseteq\Sigma\setminus\partial\Sigma$. In this paper we show that for every curve $\gamma$ on $\Sigma\setminus\punct$, the curve obtained by resolving the kinks of $\gamma$ in any order is uniquely determined, up to homotopy in $\Sigma\setminus\punct$, by the $2$-orbifold homotopy class of $\gamma$, in which the punctures are interpreted to be orbifold points of order $2$. Our proof resorts to an application of the \emph{Diamond Lemma}.
\end{abstract}

\maketitle


\tableofcontents

\section{Introduction}

Surfaces with marked points are classical mathematical objects that, after the appearance of the works \cite{fock2007dual,fomin2008cluster} and \cite{assem2010gentle,labardini2009quivers}, have suffused both cluster algebras and the representation theory of algebras during the last 15 years, with remarkable connections between geometry and representation theory discovered in works like \cite{haiden2017flat,lekili2020derived,opper2018geometric}. Recent developments \cite{amiot2022derivedequivalences,amiot2021the,labardini2022derived} have discovered that, somewhat misteriously, it is sometimes necessary to interpret the punctures not as holes, but as orbifold points of order~$2$.

When a puncture is regarded as an orbifold point of order $2$, there are certain loops, originally of infinite order, that are declared to have order $2$, namely, each loop closely wrapping around such puncture. The \emph{$2$-orbifold fundamental groupoid} of a surface $\surf$ with marked points on the boundary $\marked\subseteq \partial\Sigma$ and punctures $\punct\subseteq\Sigma\setminus\partial\Sigma$, is thus defined as the quotient groupoid of the topological fundamental groupoid of $\surf$ obtained by treating all punctures as orbifold points of order $2$.

Intuitively, a \emph{kink} of a curve $\gamma$ on $\Sigma\setminus\punct$ is a segment $\kappa$ of $\gamma$ that turns out to be one of the aforementioned loops, and that cannot be dissolved with any homotopy rel $\{0,1\}$ of curves on $\Sigma\setminus\punct$, see Figure \ref{Fig:kinkOfCurve}.
        \begin{figure}[ht]
                \caption{}\label{Fig:kinkOfCurve}
                \centering
                \includegraphics[scale=.125]{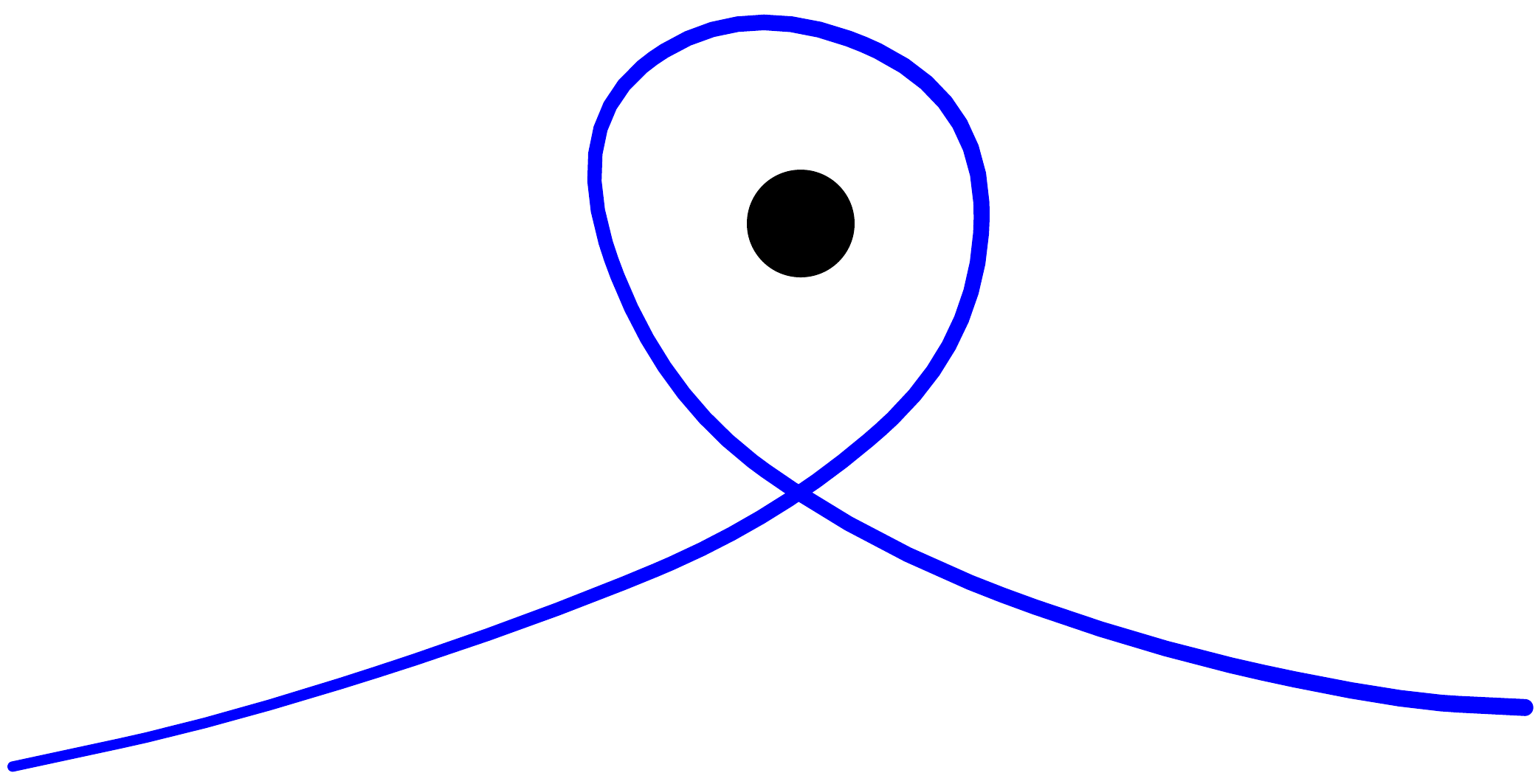}
        \end{figure}
However, the notion of \emph{kink} of an arbitrary morphism in the topological fundamental groupoid of $\surf$, that is, of an arbitrary homotopy class rel $\{0,1\}$ of curves connecting two given points on $\Sigma\setminus\punct$, is not as easy to define in terms of arbitrary representatives of the morphism as one would like, because an arbitrary representative can be a very complicated curve even if the morphism possesses some nice representative -- consider, for instance, a curve that has not only self-crossings and self-tangencies, but also different segments that it traverses multiple times, perhaps with immediate backtrackings of some of them.

In this paper, we define the notion of kink of a curve $\gamma$ in terms of the associated walk on the (leafy) dual graph of an arbitrary triangulation of signature zero. We show that $\gamma$ having kinks is independent of the triangulation of signature zero taken, and that the curve obtained by resolving the kinks of $\gamma$ in any order is uniquely determined, up to homotopy in $\Sigma\setminus\punct$, by the $2$-orbifold homotopy class of $\gamma$, in which the punctures are interpreted to be orbifold points of order $2$. Our proof resorts to a non-trivial application of the \emph{Diamond Lemma}.

Let us describe the contents of the paper in some detail. 
A \emph{triangulation of signature zero} is an ideal triangulation with the property that every puncture is enclosed by a self-folded triangle. Such a triangulation $\tau$ has its associated \emph{dual graph} $G(\tau)$, whose vertices are the triangles of $\tau$ and the boundary segments of $\surf$, with an edge connecting two triangles each time they share an arc of $\tau$, and with an edge between a boundary segment and the unique triangle containing it. To simplify the treatment of morphisms in the fundamental groupoid of the graph, we introduce the \emph{leafy dual graph} $\leafyG(\tau)$, which is a slight enlargement of $G(\tau)$. An important feature of $\leafyG(\tau)$ is that every non-identity morphism $f$ in the fundamental groupoid $\pi_1(\leafyG(\tau))$ can be uniquely represented as a backtrack-free walk on $\leafyG(\tau)$; we call such walk the \emph{standard form} of $f$.

The orientation of $\Sigma$ provides $G(\tau)$ with a natural structure of \emph{ribbon graph}, which we extend to a ribbon graph structure for $\leafyG(\tau)$. 
We use this ribbon structure of $\leafyG(\tau)$ to define the notion of \emph{kink} of any backtrack-free walk on $\leafyG(\tau)$, i.e., of any non-identity morphism in $\pi_1(\leafyG(\tau))$ written in standard form. We define the \emph{resolution} of a kink of a backtrack-free walk $w$ on $\leafyG(\tau)$ as the result of applying a purely combinatorial operation that mimics the topological replacement sketched in Figure \ref{Fig:kinkReplacement}
        \begin{figure}[ht]
                \caption{}\label{Fig:kinkReplacement}
                \centering
                \includegraphics[scale=.25]{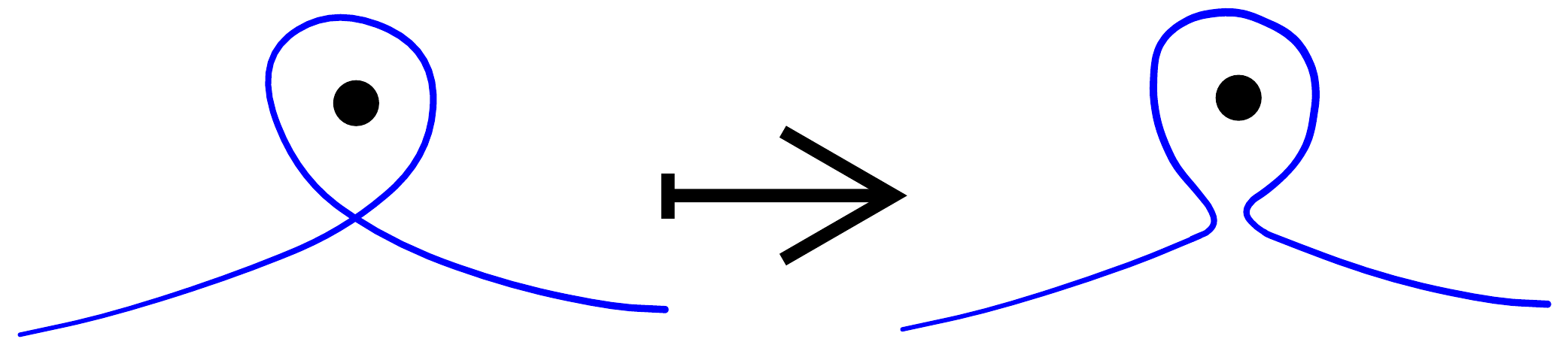}
        \end{figure}
and produces another backtrack-free walk $w'$ representing the same morphism as $w$ in the $2$-orbifold fundamental groupoid $\pi_1^{\operatorname{orb}}(\leafyG(\tau))$, but whose class in the fundamental groupoid $\pi_1(\leafyG(\tau))$ may differ from that of $w$. We prove that the number of kinks decreases every time the resolution of a kink is applied, and that if two distinct kinks of a backtrack-free walk are resolved, producing backtrack-free walks $w'$ and $w''$, then it is always possible to resolve some sequence of kinks of $w'$ and some sequence of kinks $w''$, so that the resulting backtrack-free walks coincide. This enables us to apply the \emph{Diamond Lemma} and deduce that, no matter the order in which the kinks of $w$ are resolved, one always arrives at the same kink-free backtrack-free walk without kinks after finitely many steps. 

To establish the corresponding result for curves on $\Sigma\setminus\punct$ we proceed as follows.
The \emph{ribbon surface} $\Sigma(\leafyG(\tau))$ can be naturally embedded in $\Sigma$. As usual, there are strong deformation retractions $\Sigma\setminus\punct \rightarrow\Sigma(\leafyG(\tau))\rightarrow \leafyG(\tau)$, whose composition we denote~$\rho$. We use these retractions to define the notion of \emph{kink of a curve $\gamma$ with respect to $\tau$} as a kink of the standard form of $\rho(\gamma)$. We show that for any two triangulations $\tau$ and $\sigma$ of signature zero, $\gamma$ has a kink with respect to $\tau$ if and only if $\gamma$ has a kink with respect to $\sigma$. Furthermore, the constructions and definition make it transparent that $\rho$ induces a commutative diagram of groupoids
$$
\xymatrix{\pi_1(\Sigma\setminus\punct,E) \ar[d]_{\mathfrak{p}} \ar[r]^{\rho_{\#}}_{\cong} & \pi_1(\leafyG(\tau),E) \ar[d]^{\mathfrak{p}} \\
\pi_1^{\operatorname{orb}}(\Sigma\setminus\punct,E) \ar[r]_{\overline{\rho_{\#}}}^{\cong} & \pi_1^{\operatorname{orb}}(\leafyG(\tau),E)
}
$$
whose horizotal arrows are isomorphisms, where the vertical arrows are the canonical projections of quotient groupoids. From this, the fact that the resolution of kinks of a backtrack-free walk $w$ on $\leafyG(\tau)$ does not affect the $2$-orbifold homotopy class of $w$, and the uniqueness result from the last line of the previous paragraph, we deduce the main result of this paper, namely:

\begin{theorem}
    Let $\surf$ be a surface with non-empty boundary, and let $E\subseteq \partial\Sigma$ be a set containing exactly one point from the relative interior of each boundary segment of $\surf$.
    \begin{itemize}
    \item Given $u_0,v_0\in E$, there is exactly one function $$\iota:\pi_{1,\punct}^{\operatorname{orb}}(\Sigma,E)(u_0,v_0)\rightarrow\pi_1(\Sigma\setminus\punct,E)(u_0,v_0)$$ with the following properties:
    \begin{enumerate}
        \item $\mathfrak{p}\circ\iota=\myid$;
        \item for every $f\in \pi_{1,\punct}^{\operatorname{orb}}(\Sigma,E)(u_0,v_0)$ there exists a representative curve $\gamma\in \iota(f)$ that has no kinks.
    \end{enumerate}
    \item
    there is exactly one function $$\iota:\pi_{1,\punct}^{\operatorname{orb},\operatorname{free}}(\Sigma,E)/\sim \rightarrow\pi_1^{\operatorname{free}}(\Sigma\setminus\punct,E)/\sim$$ with the following properties:
    \begin{enumerate}
        \item $\mathfrak{p}\circ\iota=\myid$;
        \item for every $f\in \pi_{1,\punct}^{\operatorname{orb},\operatorname{free}}(\Sigma,E)/\sim$ there exists a representative curve $\gamma\in \iota(f)$ that has no kinks;
    \end{enumerate}
    where $\sim$ are the equivalence relations that identify each closed curve with its opposite orientation.
    \end{itemize}
\end{theorem}

The paper is organized as follows. In Section \ref{sec:background} we present the formal definition of \emph{quotient subgroupoid} of a groupoid by what we call a \emph{normal multilocular subgroup}. We also recall the definition of the fundamental groupoid of a graph $G$, and that, provided $G$ is loop-free, every non-identity morphism $f$ in the fundamental groupoid of $G$ can be uniquely represented as a walk on $G$ without backtrackings, we call this walk the \emph{standard form} of $f$. In Section \ref{sec:leafy-dual-graph} we recall the definition of the dual graph $G(\tau)$ of a triangulation $\tau$ of signature zero, and introduce the \emph{leafy dual graph} $\leafyG(\tau)$. In Section \ref{sec:kinks-on-graphs} we define the combinatorial notions of \emph{kink} of a morfism belonging to the fundamental groupoid $\pi_1(\leafyG(\tau))$, and of \emph{resolution} of a kink, and show that the resolution of kinks satisfies the confluence conditions that allow to apply the Diamond Lemma. In Section \ref{sec:kinks-of-curves} we introduce the notion of \emph{kink of a curve} with respect to a triangulation of signature zero, and show that the absence of kinks in a curve is independent of the triangulation with respect to which it is considered. In Section \ref{sec:equiv-of-fund-groupoids} we see that (2-orbifold) fundamental groupoid of $\surf$ is isomorphic to the (2-orbifold) fundamental groupoid of $\leafyG(\tau)$. In Section \ref{sec:main-result} we present our main result.

\section{Background}\label{sec:background} 

\subsection{Groupoids, normal multilocular subgroups, and quotient groupoids}

Recall that a \emph{groupoid} is a category in which every morphism is invertible. Thus, for every object $x$ of a groupoid $\Gamma$, the endomorphism set $\Gamma(x,x)$ is a group under composition.

\begin{defi}
A \emph{multilocular subgroup} of a groupoid $\Gamma$ is a collection $H=(H(x))_{x\in\operatorname{obj}(\Gamma)}$ of subgroups $H(x)\subseteq\Gamma(x,x)$. A \emph{normal multilocular subgroup} is a multilocular subgroup $H=(H(x))_{x\in\operatorname{obj}(\Gamma)}$ with the property that for every two objects $x,y$ of $\Gamma$, every $h\in H(x)$ and every morphism
$g\in\Gamma(x,y)$, we have $ghg^{-1}\in H(y)$.
\end{defi}

Suppose that $H$ is a normal multilocular subgroup of the groupoid $\Gamma$. For each pair of objects $x,y$ of $\Gamma$, let $\equiv_H\subseteq\Gamma(x,y)\times\Gamma(x,y)$ be the relation defined by the rule
$$
f\equiv_H g \Longleftrightarrow \ \text{there exists} \ h\in H(x) \ \text{such that} \ f=gh.
$$
A routine exercise shows that $\equiv_H$ is an equivalence relation. Moreover, if $x,y,z$ are objects of $\Gamma$, and $f_1,g_1\in \Gamma(x,y)$, $f_2,g_2\in\Gamma(y,z)$ are morphisms such that $f_1\equiv_H g_1$ and $f_2\equiv_H g_2$, then, taking $h_1\in H(x)$ and $h_2\in H(y)$ such that $f_1=g_1h_1$ and $f_2=g_2h_2$, we have
$$
f_2f_1 = g_2h_2g_1h_1=g_2g_1(g_1^{-1}h_2g_1)h_1,
$$
with $(g_1^{-1}h_2g_1)h_1\in H(x)$, which shows that $f_2f_1\equiv_H g_2g_1$ too, just as in the case of groups and normal subgroups.

\begin{defi}
Suppose $\Gamma$ is a groupoid and $H$ is a normal multilocular subgroup of $\Gamma$. The \emph{quotient groupoid} $\Gamma/H$ is the category having the same objects as $\Gamma$, with
\begin{align*}
(\Gamma/H)(x,y)&:=\Gamma(x,y)/\equiv_H && \text{for} \quad x,y\in\operatorname{obj}(\Gamma),\\ 
[f_2][f_1]&:=[f_2f_1] && \text{for} \ x,y,z\in\operatorname{obj}(\Gamma) \ \text{and} \ f_1\in\Gamma(x,y),f_2\in\Gamma(y,z).
\end{align*}
\end{defi}

It is obvious that $\Gamma/H$ is indeed a groupoid, and that the canonical projection $\mathfrak{p}:\Gamma\rightarrow\Gamma/H$ is a full covariant functor, essentially surjective. Furthermore, the first isomorphism theorem is satisfied as well:

\begin{theorem}
If $\Gamma_1,\Gamma_2$ are groupoids and $F:\Gamma_1\rightarrow\Gamma_2$ is a full covariant functor, then, setting $H:=\left(\ker(\Gamma_1(x,x)\overset{F}{\rightarrow}\Gamma_2(F(x),(F(x))))\right)_{x\in \operatorname{obj}\Gamma_1}$, there is a unique covariant functor $\overline{F}:\Gamma_1/H\rightarrow \Gamma_2$ such that the diagram
$$
\xymatrix{
\Gamma_1 \ar[r]^{F} \ar[d] & \Gamma_2\\
\Gamma_1/H \ar[ur]_{\overline{F}}
}
$$
commutes. The functor $\overline{F}$ is fully faithful, and if $F$ is bijective on objects (resp. essentially surjective), then $\overline{F}$ is an isomorphism of categories (resp. an equivalence of categories).
\end{theorem}

It is straightforward to verify that if $\Gamma$ is a groupoid and $\mathcal{F}=\{H_\alpha\}_{\alpha\in I}$ is a non-empty class of normal multilocular subgroups of~$\Gamma$, then $\left(\bigcap_{\alpha\in I}H_\alpha(x)\right)_{x\in \operatorname{obj}(\Gamma)}$ is again a normal multilocular subgroup of~$\Gamma$. Hence any collection $S=(S(x))_{x\in\operatorname{obj}(\Gamma)}$ of subsets $S(x)\subseteq\Gamma(x,x)$ generates a normal multilocular subgroup $H$ of~$\Gamma$. Such $H$ has the property that for each $x\in\operatorname{obj}(\Gamma)$, the normal subgroup of $\Gamma(x,x)$ generated by $S(x)$ is contained in $H(x)$, but the containment may be proper if $H(x)$ is not a characteristic subgroup of $\Gamma(x,x)$. In any case, it is readily seen that $H(x)$ coincides with the subgroup of $\Gamma(x,x)$ generated by $\bigcup_{y\in \operatorname{obj}(\Gamma)}\{g^{-1}hg\suchthat h\in S(y), g\in \Gamma(x,y)\}$.

\subsection{The fundamental groupoid of a graph}

Let $G$ be a finite graph. As is common practice, whenever we want to think of $G$ as a topological space, we shall identify $G$ with a $1$-dimensional CW-complex having the vertices of $G$ as $0$-cells, and the edges of $G$ as $1$-cells.

Define a category $\pi_1(G)$ as follows. Its objects are the vertices of $G$. Given two objects $u,v$ of $\pi_1(G)$, we set
$$
\pi_1(G)(u,v) := P(u,v)/\simeq,
$$
where $P(u,v)$ is the set of all continuous curves from $u$ to $v$ in $G$ seen as a topological space, and $\simeq\subseteq P(u,v)\times P(u,v)$ is the equivalence relation of homotopy relative to extreme points. Composition in $\pi_1(G)$ is induced by concatenation of curves. It is well known, and easy to see, that $\pi_1(G)$ is a groupoid, the \emph{fundamental groupoid of $G$}.

\begin{defi}\label{def:backtrack-free-walks}
    Let $G$ be a loop-free graph.
    \begin{enumerate}
    \item A \emph{backtrack-free walk} on $G$ is a finite sequence $$f=(u_0,e_1,u_1,e_2,\ldots,e_{n-1},u_{n-1},e_n,u_n),$$ where $u_0,u_1,\ldots,u_n$ are vertices of $G$, and $e_1,\ldots,e_n$ are edges of $G$, such that
    \begin{itemize}
    \item for $l=1,\ldots,n$, the edge $e_l$ connects the vertices $u_l$ and $u_{l-1}$;
        \item for $l=1,\ldots,n$, the edges $e_l\neq e_{l-1}$.
    \end{itemize}
    \item A \emph{closed backtrack-free walk} is a backtrack-free walk $f=(u_0,e_1,\ldots,e_n,u_n)$ such that $u_0=u_n$ and $e_1\neq e_n$.
    \item We say that two closed backtrack-free walks $f=(u_0,e_1,u_1,\ldots,u_{n-1},e_n,u_n)$ and $g=(v_0,d_1,v_1,\ldots,v_{m-1},d_m,v_m)$ are \emph{rotationally equivalent}, and write $f\sim_{\operatorname{rot}}g$, if $n=m$ and for some index $k\in\{0,\ldots,n-1\}$ we have
    $$g=(u_k,e_{k+1},u_{k+1},\ldots,u_{n-1},e_n,u_0,e_{1},u_1,\ldots,u_{k-1},e_k,u_k).$$
    \end{enumerate}
\end{defi}

\begin{remark}\label{rem:rotational-equivalence} Let $G$ be a loop-free graph. Rotational equivalence is an equivalence relation on the set of closed backtrack-free walks. The equivalence class of such an $f$ will be denoted~$[f]_{\operatorname{rot}}$.
\end{remark}

Let $G$ be a loop-free graph.
A backtrack-free walk $(u_0,e_1,u_1,e_2,\ldots,e_{n-1},u_{n-1},e_n,u_n)$ determines a morphism from $u_0$ to $u_n$ in the fundamental groupoid $\pi_1(G)$ by concatenating the curves obtained by parameterizing each edge $e_l$ as a continuous curve from $u_{l-1}$ to $u_l$.
Similarly, each rotational equivalence class $C$ determines an element of the so-called \emph{free-homotopy fundamental group} $\pi_1^{\operatorname{free}}(G)$ by taking, for any representative $f\in C$ (thus $C=[f]_{\operatorname{rot}}$), the free-homotopy class of the closed curve $\mathbb{S}^1\rightarrow G$ defined by concatenating the curves obtained by parameterizing each edge $e_l$ as a continuous curve from $u_{l-1}$ to $u_l$.

The next result is a consequence of, e.g., \cite[Chapter~4]{may1999aconcise}.

\begin{theorem}\label{thm:morphisms-represented-by-walks-on-graphs}
    Suppose that $G$ is a loop-free graph.
    \begin{enumerate}
        \item For every pair of vertices $u_0$ and $v_0$ of~$G$, every morphism $f:u_0\rightarrow v_0$ in the fundamental groupoid $\pi_1(G)$ can be represented uniquely as a backtrack-free walk $(u_0,e_1,u_1,e_2,\ldots,e_{n-1},u_{n-1},e_n,v_0)$.
        \item Every free-homotopy class belonging to $\pi_1^{\operatorname{free}}(G)$ can be represented uniquely as a rotational equivalence class of closed backtrack-free walks on $G$.
    \end{enumerate} 
\end{theorem}

\begin{defi}\label{def:standard-form-of-a-morphism-in-a-graph-groupoid} Let $G$ is a loop-free graph, and let $f$ be either a morphism in $\pi_1(G)$ or a free-homotopy class belonging to $\pi_1^{\operatorname{free}}(G)$. The unique backtrack-free walk (resp. the unique rotational equivalence class of closed backtrack-free walks) representing $f$ in Theorem \ref{thm:morphisms-represented-by-walks-on-graphs} will be called the \emph{standard form} of $f$. Notation: $\phi(f)$.
\end{defi}

Thus, if $\phi(f)=(u_0,e_1,u_1,e_2,\ldots,e_{n-1},u_{n-1},e_n,v_0)$, then in the fundamental groupoid $\pi_1(G)$ we have the equality
$$
f=(u_0,e_1,u_1)*(u_1,e_2,u_2)*\ldots*(u_{n-2},e_{n-1},u_{n-1})*(u_{n-1},e_n,v_0),
$$
where $*$ is the operation induced by concatenation of paths (written from left to right, that is, opposite to the usual way of composing morphisms in a category, i.e., $\circ=*^{\operatorname{op}}$ for $\pi_1(G))$. In particular, $\myid_{u_0}=\phi(\myid_{u_0})=(u_0)$.

Notice that in the situation of Theorem \ref{thm:morphisms-represented-by-walks-on-graphs}, if $(u_0,e_1,\ldots,e_n,v_0)$ is the standard form of the non-identity morphism $f:u_0\rightarrow v_0$, then $(v_0,e_n,\ldots,e_1,u_0)$ is the standard form of the inverse $f^{-1}:v_0\rightarrow u_0$.

We will recur to  the following notational abuses. We will denote any given backtrack-free walk $(u_0,e_1,u_1,e_2,\ldots,e_{n-1},u_{n-1},e_n,u_n)$ of rotation class of backtrack-free walks $[(u_0,e_1,u_1,e_2,\ldots,e_{n-1},u_{n-1},e_n,u_n)]_{\operatorname{rot}}$ simply as $(u_0,e_1,e_2,\ldots,e_{n-1},e_n,u_n)$.

\begin{ex} The graph $G$ depicted in Figure \ref{Fig:exVerySmallGraph} does not have loops, so Theorem \ref{thm:morphisms-represented-by-walks-on-graphs} can be applied to it.
        \begin{figure}[!h]
                \caption{}\label{Fig:exVerySmallGraph}
                \centering
                \includegraphics[scale=.075]{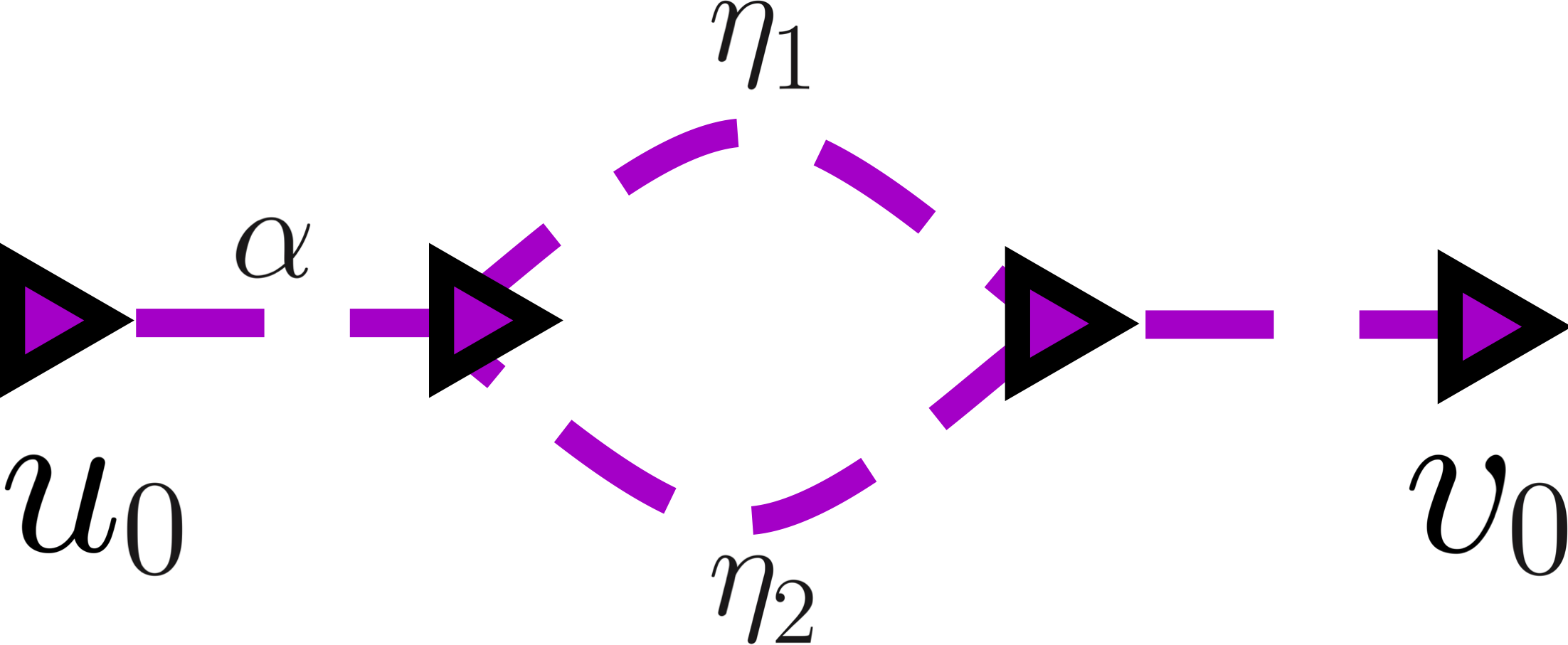}
        \end{figure}
The endomorphism group $\pi_1(G)(u_0,u_0)$ is infinite cyclic, generated by either $f=(u_0,\alpha,\eta_1,\eta_2,\alpha,u_0)$ or $f^{-1}=(u_0,\alpha,\eta_2,\eta_1,\alpha,u_0)$. For $n>0$, the standard forms of $f^n$ and $f^{-n}$ are, respectively, $$(u_0,\alpha,\underset{\operatorname{length}=2n}{\underbrace{\eta_1,\eta_2,\eta_1,\eta_2,\ldots,\eta_1,\eta_2}},\alpha,u_0) \quad \text{and} \quad (u_0,\alpha,\underset{\operatorname{length}=2n}{\underbrace{\eta_2,\eta_1,\eta_2,\eta_1,\ldots,\eta_2,\eta_1}},\alpha,u_0).$$
\end{ex}

\section{The (leafy) dual graph of a triangulation of signature zero}\label{sec:leafy-dual-graph}

The following notion is in sync with \cite[Definition 9.1 and \S9.2]{fomin2008cluster}.

\begin{defi}
Let $\surf$ be a (possibly punctured) surface with non-empty boundary.
An ideal triangulation $\tau$ of $\surf$ is said to have \emph{signature zero} if every puncture is enclosed by a self-folded triangle of $\tau$.
\end{defi}

For each ideal triangulation $\tau$ of signature zero, the \emph{dual graph} $G(\tau)$ is defined as follows. The vertices of $G(\tau)$ are the triangles of $\tau$ and the boundary segments of $\surf$. For each arc $k$ of $\tau$, we put an edge connecting the triangles that share $k$; and for each boundary segment $s$ of $\surf$, we put an edge connecting $s$ to the unique triangle of $\tau$ that contains it.
Notice that:
\begin{itemize}\item the self-folded triangles of $\tau$ are precisely the vertices incident to loops of the graph~$G(\tau)$;
\item every triangle of $\tau$ has valency $3$ as a vertex of $G(\tau)$, whereas each boundary segment has valency $1$;
\end{itemize} 

We turn $G(\tau)$ into a ribbon graph (or fat graph) in a natural way by letting the cyclic order on the edges incident to each vertex $v$ of $G(\tau)$ to be given by the clockwise sense around $v$, according to the orientation of $\Sigma$.

\begin{ex}\label{ex:triangulation-and-dual-graph} 
In Figure \ref{Fig:exDualGraph} we can see a triangulation $\tau$ of an annulus with one marked point on one boundary component, three on the other, and three punctures.
        \begin{figure}[ht]
                \caption{}\label{Fig:exDualGraph}
                \centering
                \includegraphics[scale=.5]{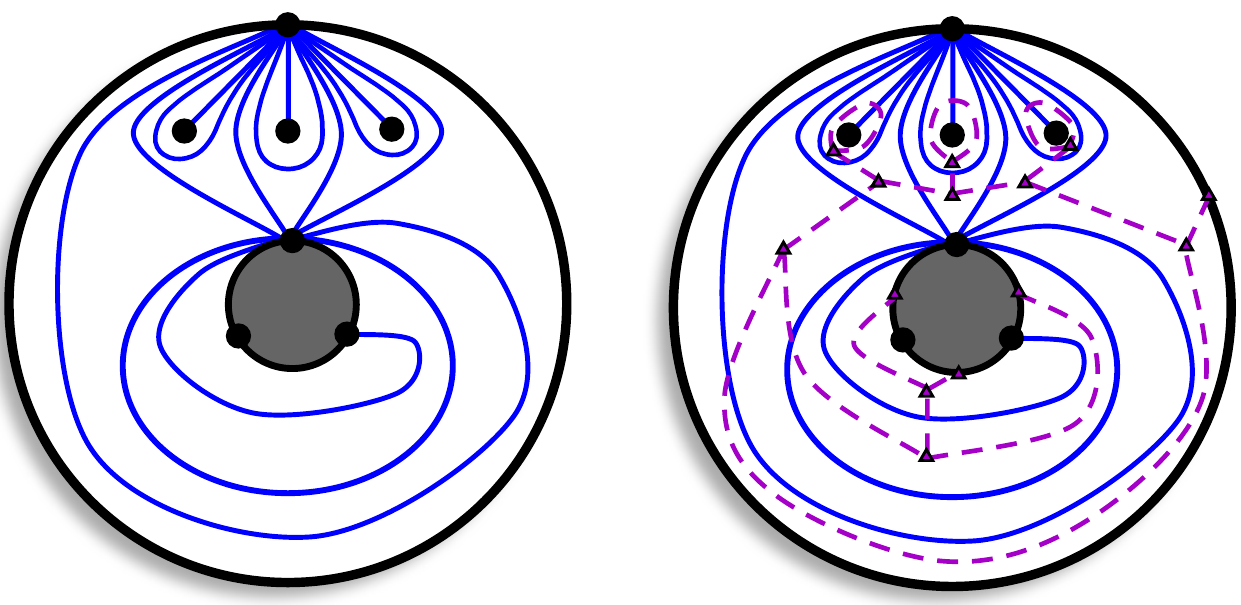}
        \end{figure}
We can also see the dual graph $G(\tau)$ drawn on the surface.
\end{ex}

\begin{defi}\label{def:leafy-dual-graph} Let $\surf$ be a (possibly punctured) surface with non-empty boundary, and $\tau$ a signature-zero ideal triangulation of $\surf$.
The \emph{leafy dual graph} of $\tau$ is the graph $\leafyG(\tau)$ obtained from $G(\tau)$ after applying the following combinatorial procedure. For each self-folded triangle $v$ of $\tau$,
\begin{enumerate}
\item split its corresponding loop $\eta_v$ into two distinct edges $\eta_{v,1}$, $\eta_{v,2}$, each connecting $v$ to a newly introduced vertex $w_v$;
\item introduce a leaf $\ell_v$ incident to $w_v$; call $z_v$ the vertex of $\ell_v$ distinct from $w_v$.
\end{enumerate}
We extend the ribbon graph structure of $G(\tau)$ to a ribbon graph structure of $\leafyG(\tau)$ as indicated in Figure \ref{Fig:ribbonStructureOffringeGtau}. 
 \begin{figure}[ht]
                \caption{}\label{Fig:ribbonStructureOffringeGtau}
                \centering
                \includegraphics[scale=.08]{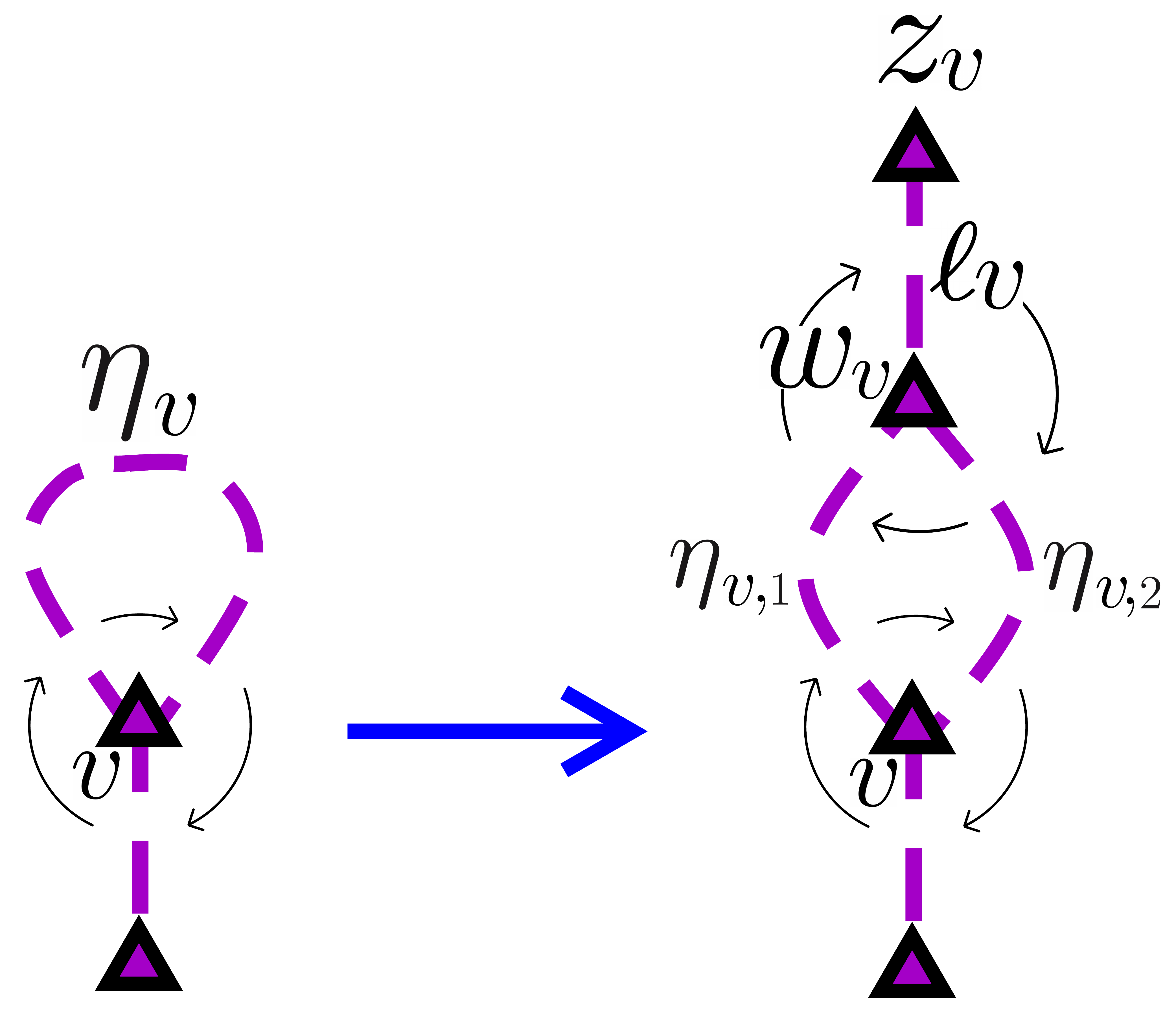}
        \end{figure}
\end{defi}

\begin{remark}\label{rem:special-vertices-of-fringeGtau}
We stress the fact that all the vertices of $G(\tau)$ are vertices of $\leafyG(\tau)$ as well, and that every boundary segment of $\surf$ is a vertex of $\tau$.
\end{remark}

Every edge of the leafy dual graph $\leafyG(\tau)$ is incident to a vertex of valency 3. Given a valency-$3$ vertex $u$ of $\leafyG(\tau)$ and an edge $e$ of $\leafyG(\tau)$ containing $u$, we use the ribbon graph structure of $\leafyG(\tau)$ to define
\begin{align}\label{eq:+turn-and--turn}
e^{+,u} &:= \text{the edge incident to $u$ which is preceded by $e$ around $u$};\\
\nonumber
e^{-,u} &:= \text{the edge incident to $u$ which is followed by $e$ around $u$}.
\end{align}

\begin{ex} Consider the triangulation $\tau$ from Example \ref{ex:triangulation-and-dual-graph}. On the left hand side of Figure \ref{Fig:exfringeGraph} we can see the ribbon graph $G(\tau)$.
        \begin{figure}[ht]
                \caption{}\label{Fig:exfringeGraph}
                \centering
                \includegraphics[scale=.55]{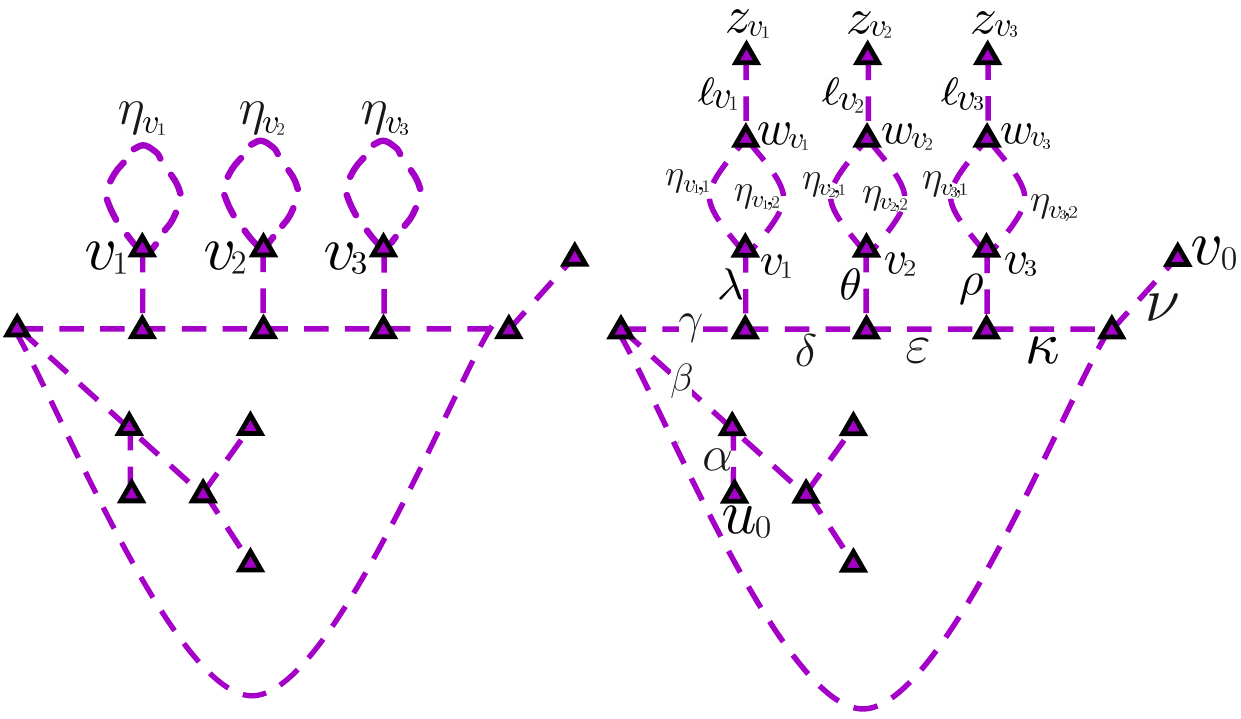}
        \end{figure}
On the right hand side of the figure, we can see the ribbon graph $\leafyG(\tau)$.
\end{ex}

Although the next result is quite standard, we briefly sketch its proof, essentially following \cite[\S2.3]{amiot2016the}, \cite[\S2.2]{amiot2016derived} and \cite[\S3.2]{amiot2021derived}. Along the way, we emphasize that the standard construction of the ribbon surface $\ribsurf{\leafyG(\tau)}$ associated to $\leafyG(\tau)$ allows to see it as the result of gluing very rigid hexagonal ribbons. This will be very useful later on, to fix a very specific class of curves for which the notion of `kink' is easy to define, and with the property that every morphism in the fundamental groupoid $\pi_1(\ribsurf{\leafyG(\tau)})$ is represented by at least one curve in the class.

\begin{theorem}\label{thm:strong-def-retraction-to-leafy-dual-graph}
    The leafy dual graph $\leafyG(\tau)$ is a strong deformation retract of $\Sigma\setminus\punct$. 
\end{theorem}

\begin{proof} 
    For each edge $e$ of $\leafyG(\tau)$, let $R_e$ be an open regular oriented Euclidean hexagon inscribed in the unit circle in the complex plane, and take an embedding of $e$ into $R_e$ as a Euclidean straight line segment joining a pair of radially opposite vertices. We fix such an $R_e$ and such an embedding of $e$ into $R_e$ once and for all. See Figure \ref{Fig:ribSurfAndRetraction} (left).
        \begin{figure}[ht]
                \caption{}\label{Fig:ribSurfAndRetraction}
                \centering
                \includegraphics[scale=.175]{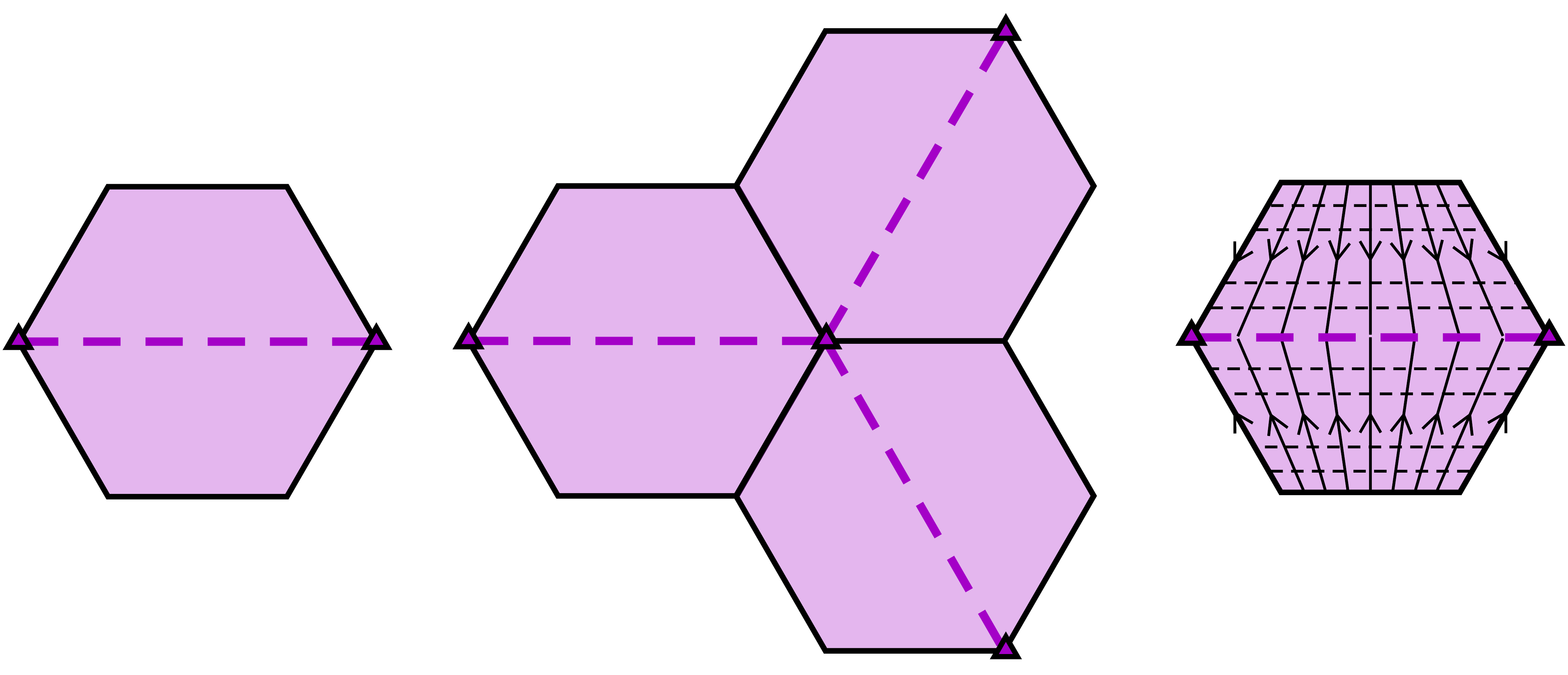}
        \end{figure}

    For each valency-$3$ vertex $u$ of $\leafyG(\tau)$, take an edge $e$ of $\leafyG(\tau)$ containing $u$, and glue the hexagons $R_e$, $R_{e^{+,u}}$ and $R_{e^{-,u}}$ as indicated in Figure \ref{Fig:ribSurfAndRetraction} (center), see also \eqref{eq:+turn-and--turn}. Denote by $\ribsurf{\leafyG(\tau)}$ the result of performing this gluing over all valency-$3$ vertices of $\leafyG(\tau)$. We shall refer to $\ribsurf{\leafyG(\tau)}$ as the \emph{ribbon surface} of $\leafyG(\tau)$. 
     We can take a natural open, piecewise-differentiable, continuous, injective function 
     \begin{equation}\label{eq:embedding-of-ribbon-surfaces}
     \iota_\tau:\Sigma(\leafyG(\tau))\rightarrow \Sigma\setminus(\punct\cup \partial\Sigma)
     \end{equation}
    that embeds $\Sigma(\leafyG(\tau))$ as an open subsurface of $\Sigma\setminus\punct$. Similarly to \cite[\S2.3]{amiot2016the}, \cite[\S2.2]{amiot2016derived} and \cite[\S3.2]{amiot2021derived},  
    is not hard to see that this embedding admits a strong deformation retraction $\Sigma\setminus\punct\rightarrow \ribsurf{\leafyG(\tau)}$.
    
    For each edge $e$ of $\leafyG(\tau)$, let $\varrho_e:R_e\rightarrow e$ be the piecewise-linear strong deformation retraction sketched in Figure \ref{Fig:ribSurfAndRetraction} (right). Define $\varrho:\ribsurf{\leafyG(\tau)}\rightarrow \leafyG(\tau)$ by setting $\varrho|_{R_e}:=\varrho_e$ for every edge $e$. Then $\varrho$ is a strong deformation retraction. Composing $\varrho$ with a strong deformation retraction $\Sigma\setminus\punct\rightarrow \ribsurf{\leafyG(\tau)}$ from the previous paragraph, we obtain a strong deformation retraction $\rho:\Sigma\setminus\punct\rightarrow\leafyG(\tau)$.
\end{proof}

\begin{ex}
For the ribbon graph $\leafyG(\tau)$ from Figure \ref{Fig:exfringeGraph}, the ribbon surface $\ribsurf{\leafyG(\tau)}$ can be visualized in Figure \ref{Fig:exRibbonSurface}.
        \begin{figure}[ht]
                \caption{}\label{Fig:exRibbonSurface}
                \centering
                \includegraphics[scale=.4]{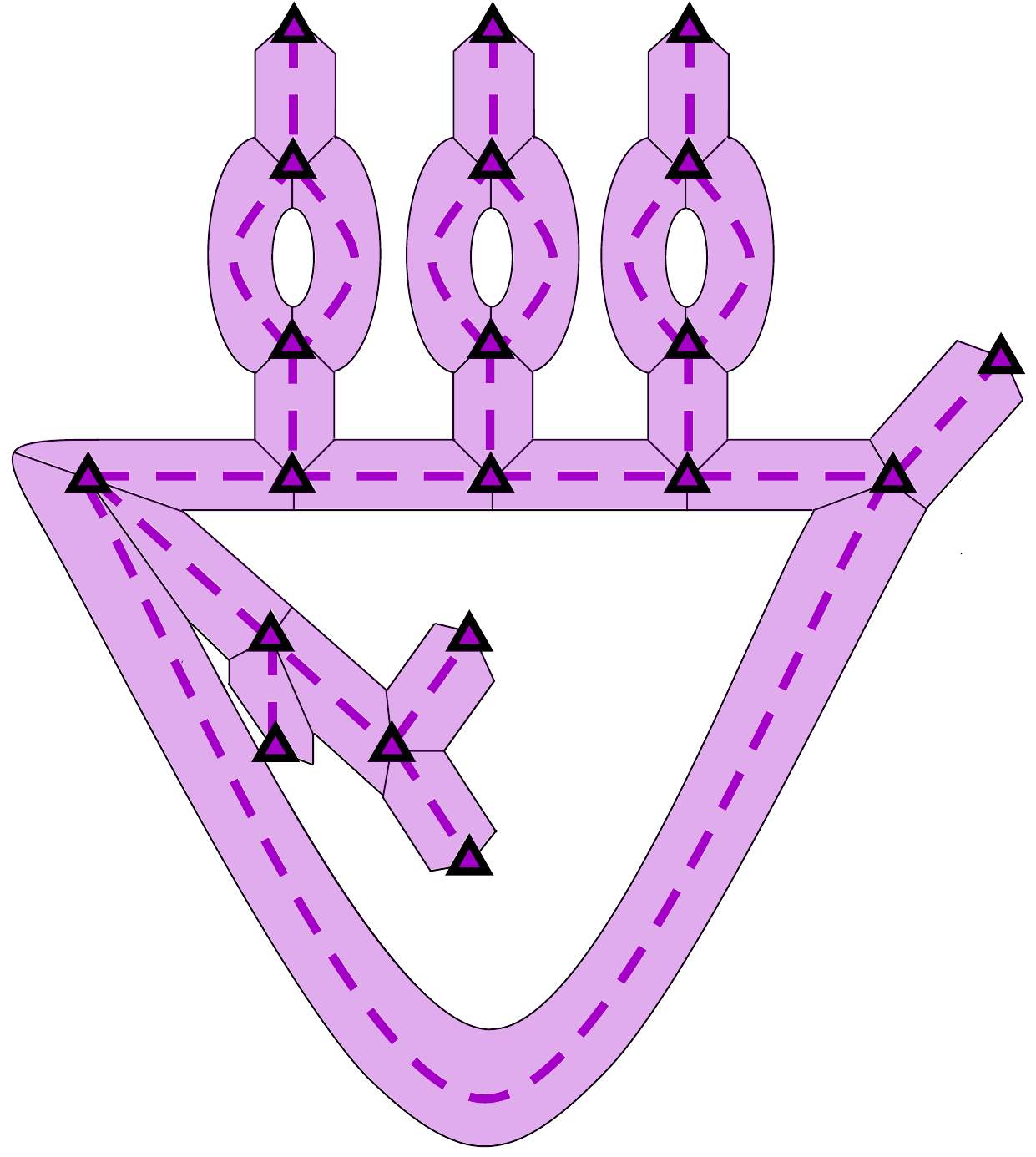}
        \end{figure}
Strictly speaking, the picture is incorrect, since the hexagonal ribbons $R_e$ are not drawn as regular hexagons congruent to each other. Drawing them as such would interfere with an intuitive visualization of $\ribsurf{\leafyG(\tau)}$.
\end{ex}

\section{Kinks of walks on graphs}\label{sec:kinks-on-graphs}

From this point to the end of the paper, we will work only with triangulations of signature zero.
By construction, no edge of $\leafyG(\tau)$ is a loop. Hence Theorem \ref{thm:morphisms-represented-by-walks-on-graphs} can be applied. Let $f$ be either
\begin{align}\label{eq:f-morfism-in-pi_1(leafyG(tau))}
&
\text{a non-identity morphism in $\pi_1(\leafyG(\tau))$ between vertices of $G(\tau)$ }\\
\nonumber &
\text{distinct from all the self-folded triangles of $\tau$,
}  \\
\nonumber & \text{or}\\
\label{eq:free-homotopy-class-in-pi_1(leafyG(tau))}
&\text{a non-contractible free-homotopy class belonging to $\pi_1^{\operatorname{free}}(\leafyG(\tau))$.} 
\end{align}
We define a finite sequence $\mathbf{s}(f)=(\varepsilon_1,\varepsilon_2,\ldots)$ of signs $\varepsilon_j\in\{+,-\}$ as follows. Let $\phi(f)=(u_0,e_1,\ldots,e_n,v_0)$ be the standard form of $f$. If $n=1$, we set $\mathbf{s}(f)$ to be the empty sequence. Else, for $j=1,\ldots,n-1$, set $\varepsilon_j$ to be the sign with the property that
$$
e_{j+1}=e_{j}^{\varepsilon_{j},u_j}.
$$
If $f$ belongs to $\pi_1^{\operatorname{free}}(\leafyG(\tau))$, we further set $\varepsilon_n$ to be the sign with the property that
$$
e_{1}=e_{n}^{\varepsilon_{n},u_n}.
$$
We define 
$$
\mathbf{s}(f):=\begin{cases}
(\varepsilon_1,\ldots,\varepsilon_{n-1}) & \text{if $f$ is a morphism in $\pi_1(G)$};\\
(\varepsilon_1,\ldots,\varepsilon_{n-1},\varepsilon_n) & \text{if $f$ belongs to $\pi_1^{\operatorname{free}}(\leafyG(\tau))$.}
\end{cases}
$$
Notice that if $f\in\pi_1^{\operatorname{free}}(\leafyG(\tau))$, then $\mathbf{s}(f)$ is defined up to a cyclic rotation of its entries.

\begin{ex}\label{ex:sign-sequences-on-leafy-graph} Consider the vertices $u_0,v_0$ of the graph $\leafyG(\tau)$ shown in Figure \ref{Fig:exfringeGraph}. 
For the following morphisms $f:u_0\rightarrow v_0$ (written in standard form) in $\pi_1(\leafyG(\tau))$, we can see the sign sequence $\mathbf{s}(f)$ on the right (the commas have been omitted for space reasons).
\begin{align*}
    f_1 &= (u_0\alpha\beta\gamma\delta\varepsilon\kappa\nu v_0) & \mathbf{s}(f_1) &= (+----+) \\
    f_2 &= (u_0\alpha\beta\gamma\lambda\eta_{v_1,1}\eta_{v_1,2}\lambda\delta\varepsilon\kappa\nu v_0) & \mathbf{s}(f_2) &= (+-++-++--+)\\
    f_3 &= (u_0\alpha\beta\gamma\lambda\eta_{v_1,1}\eta_{v_1,2}\eta_{v_1,1}\eta_{v_1,2}\lambda\delta\varepsilon\kappa\nu v_0) & \mathbf{s}(f_3) &= (+-++---++--+)\\
    f_4 &= (u_0\alpha\beta\gamma\lambda\eta_{v_1,2}\eta_{v_1,1}\lambda\delta\varepsilon\kappa\nu v_0) & \mathbf{s}(f_4) &= (+-+-+-+--+)
\end{align*}
\end{ex}

\begin{defi}\label{def:what-is-a-kink-on-a-graph} Let $\tau$ be a triangulation of signature zero, and $\leafyG(\tau)$ its leafy dual graph.
\begin{enumerate}
\item If $f$ is a morfism as in \eqref{eq:f-morfism-in-pi_1(leafyG(tau))}, then 
a \emph{kink} of $f$ is a segment $(e_{j},e_{j+1},\ldots,e_{m})$ of the standard form $\phi(f)=(u_0,e_1,\ldots,e_n,u_n)$, with the property that $r:=\frac{m-j-3}{2}$ is a positive integer and there exists a self-folded triangle $v$ of $\tau$, and distinct indices $i,k\in\{1,2\}$ such that either
\begin{enumerate}
    \item $e_{j+r+1}=\eta_{v,i}$ and $e_{j+r+2}=\eta_{v,k}$;
    \item for $t=1,\ldots,r$, we have $e_{j+t}=e_{m-t}$;
    \item $e_{j+r}=e_{j+r+3}\notin\{\eta_{v,k},\eta_{v,i}\}$; and
    \item\label{item:what-is-a-kink-on-a-graph-crucial-signs-multiplicity-1} $\varepsilon_{j}=\varepsilon_{j+r+1}=\varepsilon_{m-1}$;
\end{enumerate}
or
\begin{itemize}
    \item[(i)] $e_{j+1}=e_{j+3}=\ldots=e_{j+2r+1}=\eta_{v,i}$ and $e_{j+2}=e_{j+4}=\ldots=e_{j+2r+2}=\eta_{v,k}$;~and
    \item[(ii)] $e_{j}=e_{m}\notin \{\eta_{v,i},\eta_{v,k}\}$.
\end{itemize}
\item if $f$ is a free-homotopy class as in \eqref{eq:free-homotopy-class-in-pi_1(leafyG(tau))}, then a \emph{kink} of $f$ is a kink of any representative of the rotational equivalence class of the standard form $\phi(f)$.
\end{enumerate}
We say that a kink has \emph{multiplicity} $1$ or $r+1$ according to whether it satisfies the first or second set of conditions above. Furthermore, the \emph{core} of $\kappa$ is its segment $(e_{j+r+1},e_{j+r+2})$ if the multiplicity of $\kappa$ is $1$, and its segment $(e_{j+1},e_{j+2},\ldots,e_{j+2r+1},e_{j+2r+2})$ if the multiplicity of $\kappa$ is greater than $1$.
\end{defi}

\begin{ex} With reference to Example \ref{ex:sign-sequences-on-leafy-graph}, the morphisms $f_1$ and $f_2$ do not have kinks whatsoever. The morphisms $f_3$ and $f_4$ 
do have kinks (again, for space reasons we omit the commas):
\begin{align*}
    f_3 &= (u_0\alpha\beta\gamma\underset{\operatorname{mult}=2,j=4,m=9,r=1}{\underbrace{\lambda\eta_{v_1,1}\eta_{v_1,2}\eta_{v_1,1}\eta_{v_1,2}\lambda}}\delta\varepsilon\kappa\nu v_0) & \mathbf{s}(f_3) &= (+-++---++--+)\\
    f_4 &= (u_0\alpha\beta\underset{\operatorname{mult}=1,j=3,m=8,r=1}{\underbrace{\gamma\lambda\eta_{v_1,2}\eta_{v_1,1}\lambda\delta}}\varepsilon\kappa\nu v_0) & \mathbf{s}(f_4) &= (+-\boxed{+}-\boxed{+}-\boxed{+}--+) 
\end{align*}
For the morphism $f_4$, the signs $\varepsilon_{j}=\varepsilon_{j+r+1}=\varepsilon_{m-1}$ from Definition \ref{def:what-is-a-kink-on-a-graph}-\eqref{item:what-is-a-kink-on-a-graph-crucial-signs-multiplicity-1} appear enclosed in squares.
\end{ex}

\begin{defi}\label{def:partial-resolution-of-kink}
Let $\tau$ be a triangulation of signature zero, and $\leafyG(\tau)$ its leafy dual graph. 
\begin{enumerate}
\item 
Suppose that $f$ is a morphism as in \eqref{eq:f-morfism-in-pi_1(leafyG(tau))}, and that $\kappa=(e_{j},e_{j+1},\ldots,e_{m})$ is a kink of $f$, and set $r:=\frac{m-j+3}{2}$ as in Definition \ref{def:what-is-a-kink-on-a-graph}. A \emph{partial resolution} of $\kappa$ in $f$ is the sequence $\widetilde{\rho}_{\kappa}(f)$ obtained after applying to $f$ one of the following two combinatorial operations:
\begin{itemize}
\item If $\kappa$ has multiplicity $1$, switch the $(j+r+1)^{\operatorname{th}}$ and $(j+r+2)^{\operatorname{th}}$ entries of $f$.
\item If $\kappa$ has multiplicity $r+1>1$, choose an index $s\in\{j+1,\ldots,j+2r+1\}$ and switch the $s^{\operatorname{th}}$ and $(s+1)^{\operatorname{th}}$ entries of $f$.
\end{itemize}
\item Suppose that $f$ is a free-homotopy class as in \eqref{eq:free-homotopy-class-in-pi_1(leafyG(tau))}, and that $\kappa=(e_{j},e_{j+1},\ldots,e_{m})$ is a kink of $f$. Suppose further that $\phi(f)=\{f_1,\ldots,f_n\}$ (see Definitions \ref{def:standard-form-of-a-morphism-in-a-graph-groupoid} and \ref{def:what-is-a-kink-on-a-graph}, as well as Remark \ref{rem:rotational-equivalence}), and that $\kappa$ appears as a kink of $f_j\in\phi(f)$. A \emph{partial resolution} of $\kappa$ in $f$ is the sequence $\widetilde{\rho}_{\kappa}(f_j)$ obtained by applying to $f_j$ one of the two combinatorial operations just described.
\end{enumerate}
\end{defi}

Suppose that $f$ is as in \eqref{eq:f-morfism-in-pi_1(leafyG(tau))} or \eqref{eq:free-homotopy-class-in-pi_1(leafyG(tau))}, and $\kappa$ is a kink of $f$.
If $f$ is a morphism and $\kappa$ has multiplicity $1$, then the sequence $\widetilde{\rho}_{\kappa}(f)$ is well defined, i.e., uniquely determined by $f$ and $\kappa$, whereas if $k$ has multiplicity $r+1>1$, then there is ambiguity in the definition of $\widetilde{\rho}_{\kappa}(f)$ as a sequence. Furthermore, if $f$ is a free-homotopy class, then $\kappa$ does not appear as a kink of all the representatives $f_1,\ldots,f_n$, as rotation of walks eventually breaks the appearance of $\kappa$ as a subsequence. However, we do have the following result, whose proof is immediate:

\begin{proposition}\label{prop:contraction-after-partial-resolution-of-kink}
Let $\tau$ be a triangulation of signature zero, and $\leafyG(\tau)$ its leafy dual graph, and let $f$ be as in \eqref{eq:f-morfism-in-pi_1(leafyG(tau))} or \eqref{eq:free-homotopy-class-in-pi_1(leafyG(tau))}.
\begin{enumerate}
\item If $f$ is a morphism, then for any kink $\kappa$ of $f$, the sequence $\phi(\widetilde{\rho}_{\kappa}(f))$ is uniquely determined by $f$ and $\kappa$. Thus, for any kink $\kappa$ of multiplicity greater than $1$, the sequence $\phi(\widetilde{\rho}_{\kappa}(f))$ is independent of the index $s$ chosen in Definition \ref{def:partial-resolution-of-kink}. 
\item If $f$ is a free-homotopy class, with standard form $\phi(f)=\{f_1,\ldots,f_n\}$, and $\kappa$ appears as a kink of $f_i$ and $f_j$, with such appearances being brought to each other by the corresponding rotations that bring $f_i$ and $f_j$ to each other, then $[\phi(\widetilde{\rho}_{\kappa}(f_i))]_{\operatorname{rot}}=[\phi(\widetilde{\rho}_{\kappa}(f_j))]_{\operatorname{rot}}$.
\end{enumerate}
Furthermore, the number of kinks of $\phi(\widetilde{\rho}_{\kappa}(f))$ (resp. $[\phi(\widetilde{\rho}_{\kappa}(f_j))]_{\operatorname{rot}}$), counted with multiplicity, is strictly less than the number of kinks of $f$, counted with multiplicity.
\end{proposition}

\begin{defi}\label{def:resolution-of-kink}
In the situation of Proposition \ref{prop:contraction-after-partial-resolution-of-kink}, 
$$\rho_\kappa(f):=\begin{cases}\phi(\widetilde{\rho}_\kappa(f)) & \text{if $f$ is a morphism};\\
[\phi(\widetilde{\rho}_{\kappa}(f_j))]_{\operatorname{rot}} & \text{if $f$ is a free-homotopy class}
\end{cases}$$ is the \emph{resolution} of $\kappa$ in~$f$.
\end{defi}

\begin{lemma}\label{lemma:kink-resolution-makes-multiplicity-drop}
Let $\tau$ be a triangulation of signature zero, and $\leafyG(\tau)$ its leafy dual graph, and let $f$ be as in \eqref{eq:f-morfism-in-pi_1(leafyG(tau))} or \eqref{eq:free-homotopy-class-in-pi_1(leafyG(tau))}.
\begin{itemize}
\item Resolving a kink of multiplicity $1$ makes it disappear but does not affect the length of $f$.
\item Resolving a kink of multiplicity $2$ makes it disappear and makes the length of $f$ drop by at least $4$.
\item Resolving a kink $\kappa$ of $f$ of multiplicity $3$ makes the multiplicity of $\kappa$ drop by $2$ or $3$ and the length of $f$ drop by $4$. 
\item Resolving a kink $\kappa$ of $f$ of multiplicity greater than $3$ makes the multiplicity of $\kappa$ drop by $2$ and the length of $f$ drop by $4$. 
\end{itemize}
\end{lemma}

The proof of Lemma \ref{lemma:kink-resolution-makes-multiplicity-drop} is left to the reader.

\begin{lemma}\label{lemma:key-lemma-for-confluence-of-kink-resolutions} Let $u_0,v_0,w_0$ be vertices of $G(\tau)$ distinct from all the self-folded triangles of $\tau$.
Suppose $f=(u_0,e_1,e_2,\ldots,e_n,v_0)$ and $g=(v_0,e_n,\ldots,e_{2},d_1,w_0)$ are non-identity morphisms in $\pi_1(\leafyG)$, both written in standard form, with $e_1\neq d_1$, and that $\kappa=(e_{j},\ldots,e_{m})$ is a kink of $(e_2,\ldots,e_n)$. 
\begin{enumerate}
\item If the multiplicity of $\kappa$ is $1$, then for $r:=\frac{m-j-3}{2}$ we have
$$\phi(\rho_{\kappa}(f)*g)=(u_0,e_1,\ldots,e_{j+r-1},\underset{\text{entries $j+r$ to $j+r+5$}}{\underbrace{e_{j+r},e_{j+r+2},e_{j+r+1},e_{j+r+2},e_{j+r+1},e_{j+r}}},,e_{j+r-1},\ldots,d_1,w_0);$$
\item if the multiplicity of $\kappa$ is greater than $1$, then
$$\phi(\rho_{\kappa}(f)*g)=(u_0,e_1,e_2,\ldots,e_{j-1},\underset{\text{entries $j$ to $j+5$}}{\underbrace{e_{j},e_{j+1},e_{j+2},e_{j+1},e_{j+2},e_{j}}},,e_{j-1},\ldots,e_2,d_1,w_0);$$
\end{enumerate}
In any case, 
$\phi(\rho_{\kappa}(f)*g)$ has a kink $\kappa'$ of multiplicity $2$, such that
$$\phi(f*g)=(u_0,e_1,d_1,,w_0)
=\rho_{\kappa'}(\phi(\rho_{\kappa}(f)*g)).$$
\end{lemma}

\begin{proof} In the case where the multiplicity of $\kappa$ is $1$, the computation of $\phi(\rho_{\kappa}(f)*g)$ can be seen in Figure \ref{Fig:keyLemmaMult1},
            \begin{figure}[ht]
                \caption{}\label{Fig:keyLemmaMult1}
                \centering
                \includegraphics[scale=.35]{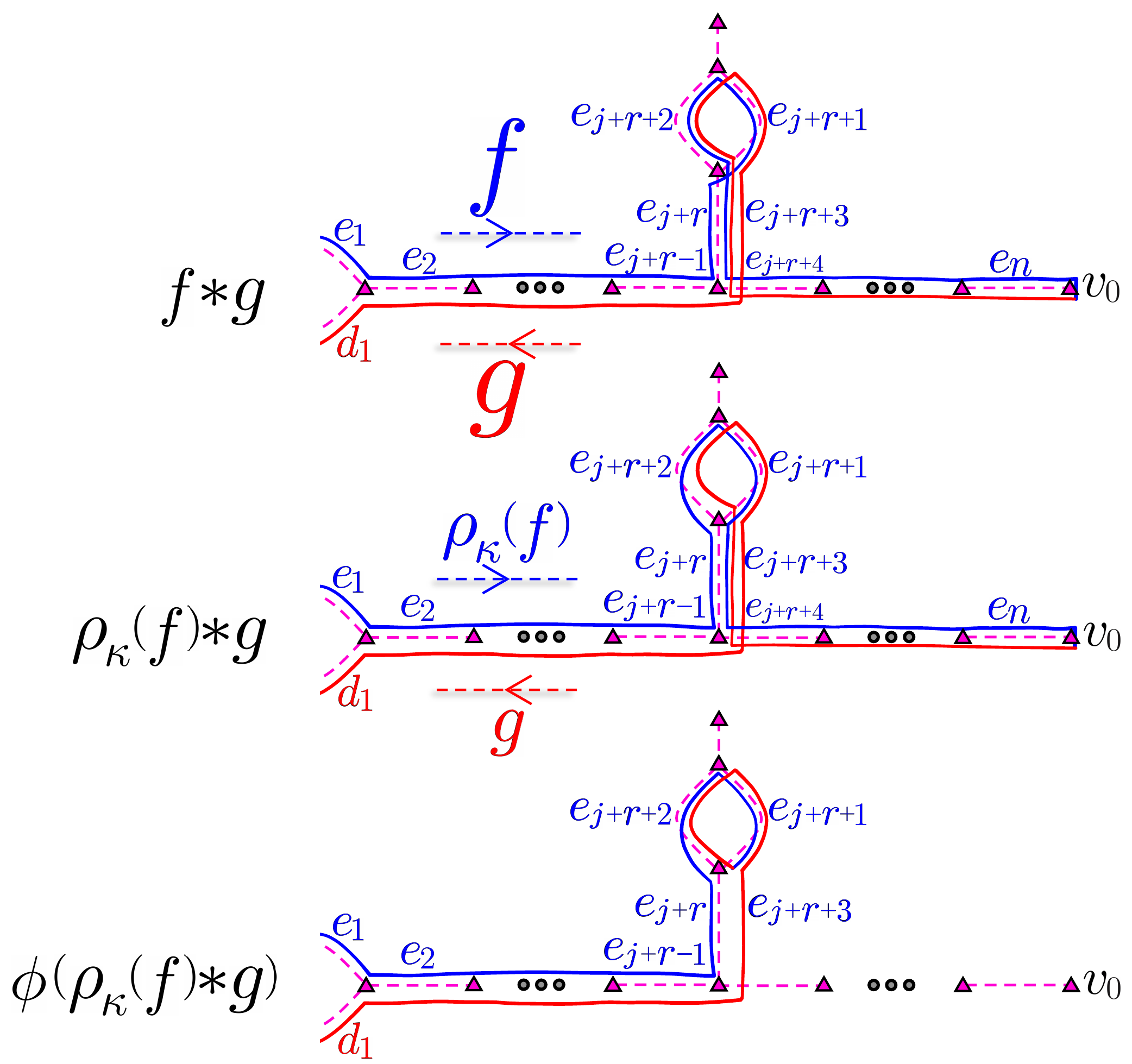}
        \end{figure}
    whereas
    in Figure \ref{Fig:keyLemmaMult3} we have included the computation of $\phi(\rho_{\kappa}(f)*g)$ in the case where the multiplicity of $\kappa$ is exactly $3$.
                    \begin{figure}[ht]
                \caption{}\label{Fig:keyLemmaMult3}
                \centering
                \includegraphics[scale=.35]{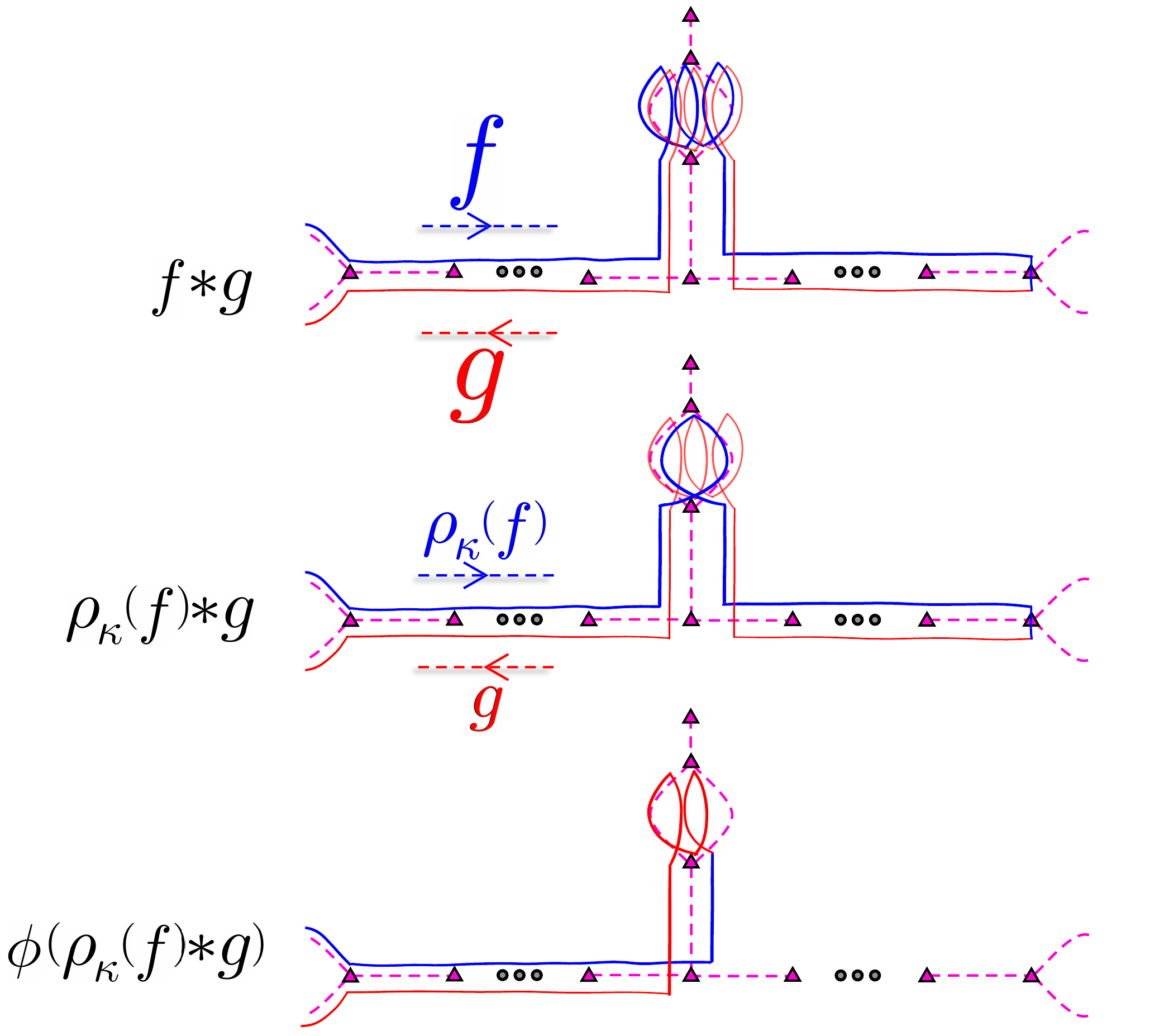}
        \end{figure}
     The computation $\phi(\rho_{\kappa}(f)*g)$ in the case where the multiplicity of $\kappa$ is an arbitrary integer greater than $1$ is completely analogous.

With $\phi(\rho_{\kappa}(f)*g)$ explicitly computed, the assertions of the lemma can be directly verified.
\end{proof}

\begin{proposition}\label{prop:local-confluence-of-kink-resolutions}
Let $\tau$ be a triangulation of signature zero, and $\leafyG(\tau)$ its leafy dual graph, and let $\varphi$ be as in \eqref{eq:f-morfism-in-pi_1(leafyG(tau))} or \eqref{eq:free-homotopy-class-in-pi_1(leafyG(tau))}.
If $\kappa_0$ and $\kappa_0'$ are distinct kinks of $\varphi$, then there exist finite sequences of kinks
    $(\kappa_1,\ldots,\kappa_{\ell})$ and $(\kappa_1',\ldots,\kappa_{\ell'}')$ such that:
    \begin{itemize}
        \item[(1)] for each $t=1,\ldots,\ell$, $\kappa_t$ is a kink of $\rho_{\kappa_{t_1}}\rho_{\kappa_{t-2}}\cdots\rho_{\kappa_0}(\varphi)$;
        \item[(2)] for each $t=1,\ldots,\ell'$, $\kappa'_t$ is a kink of $\rho_{\kappa_{t_1}'}\rho_{\kappa_{t-2}'}\cdots\rho_{\kappa_0'}(\varphi)$; 
        \item[(3)] $\rho_{\kappa_\ell}\cdots\rho_{\kappa_1}\rho_{\kappa_0}(\varphi)=\rho_{\kappa'_{\ell'}}\cdots\rho_{\kappa'_1}\rho_{\kappa'_0}(\varphi)$ if $\varphi$ is a morphism; 
        \item[(3')] $\rho_{\kappa_\ell}\cdots\rho_{\kappa_1}\rho_{\kappa_0}(\varphi)$ is equal to $\rho_{\kappa'_{\ell'}}\cdots\rho_{\kappa'_1}\rho_{\kappa'_0}(\varphi)$ or its opposite orientation if $\varphi$ is a free-homotopy class as in \eqref{eq:free-homotopy-class-in-pi_1(leafyG(tau))}.
    \end{itemize}
\end{proposition}

\begin{proof} Throughout the proof we will assume, without loss of generality, that $\varphi$ is already written in standard form.

Suppose first that $\varphi$ is a morphism as in \eqref{eq:f-morfism-in-pi_1(leafyG(tau))}. 
We distinguish five cases, namely:
    \begin{itemize}
        \item[(i)] when each of $\kappa_0$ and $\kappa_0'$ has multiplicity greater than $2$;
        \item[(ii)] when one of $\kappa_0$ and $\kappa_0'$ has multiplicity greater than $2$ and the other one has multiplicity equal to $2$;
        \item[(iii)] when each of $\kappa_0$ and $\kappa_0'$ has multiplicity equal to $2$;
        \item[(iv)] when one of $\kappa_0$ and $\kappa_0'$ has multiplicity equal to $2$ and the other one has multiplicity equal to $1$;
        \item[(v)] when each of $\kappa_0$ and $\kappa_0'$ has multiplicity equal to $1$.
    \end{itemize}

\begin{case}
    Using Lemma \ref{lemma:kink-resolution-makes-multiplicity-drop}, it is easy to verify that if each of $\kappa_0$ and $\kappa_0'$ has multiplicity greater than $2$, then $\kappa_0$ is a kink of $\rho_{\kappa_0'}(\varphi)$, $\kappa_0'$ is a kink of $\rho_{\kappa_0}(\varphi)$, and $\rho_{\kappa_0}(\rho_{\kappa_0'}(\varphi))=\rho_{\kappa_0'}(\rho_{\kappa_0}(\varphi))$, so the proposition follows in that case.
\end{case}

\begin{case}\label{case:mult2-mult>2}
Suppose that the multiplicity of $\kappa_0$ is $2$ and the multiplicity of $\kappa_0'$ is greater than~$2$. Then $\kappa_0$ is a kink of $\rho_{\kappa_0'}(\varphi)$ by Lemma \ref{lemma:kink-resolution-makes-multiplicity-drop}, and $\varphi$ can be written as a concatenation
$$
\varphi=
\alpha*f*\kappa_0*g*\beta,
$$
where $\alpha$, $f$, $g$ and $\beta$ are morphisms in $\leafyG(\tau)$ between vertices of $G(\tau)$ distinct from all the self-folded triangles of $\tau$, with $f=(u_0,e_1,e_2,\ldots,e_n,v_0)$ and $g=(v_0,e_n,\ldots,e_{2},d_1,w_0)$, and either
$$
\begin{cases}
\text{$e_1 = d_1$ and $\alpha=\myid_{u_0}=\beta$}, & \text{or} \\
e_1 \neq d_1.
\end{cases}
$$

If $e_1\neq d_1$, then either $\kappa_0'$ is disjoint from $f$ and from $g$, or it overlaps one of them. If $\kappa_0'$ is entirely contained in $f$, then we can apply Lemma \ref{lemma:key-lemma-for-confluence-of-kink-resolutions} to obtain a kink $\kappa_1'$ of $\rho_{\kappa_0}\rho_{\kappa_0'}(\varphi)=\phi(\alpha*\rho_{\kappa_0'}(f)*g*\beta)$ such that
$$
\rho_{\kappa_0}(\varphi)=\phi(\alpha*f*g*\beta)=\rho_{\kappa_1'}\rho_{\kappa_0}\rho_{\kappa_0'}(\varphi),
$$
which is the assertion of the proposition.
Essentially the same argument shows that the proposition holds also when $\kappa_0'$ either is entirely contained in $g$, or is disjoint from $f$ and from $g$, or overlaps one of them without being entirely contained in either. 

The same argument works also if $e_1 = d_1$ and $\alpha=\myid_{u_0}=\beta$; notice that in this case we have $\rho_{\kappa_0}(\varphi)=\myid_{u_0}$.
\end{case}

\begin{case} Suppose that each of $\kappa_0$ and $\kappa_0'$ has multiplicity equal to $2$. Then $\varphi$ can be written as concatenations
$$
\alpha*f*\kappa_0*g*\beta=\varphi =\alpha'*f'*\kappa_0'*g'*\beta',
$$
where $\alpha$, $\alpha'$, $f$, $f'$, $g$, $g'$, $\beta$ and $\beta'$ are morphisms in $\leafyG(\tau)$ between vertices of $G(\tau)$ that are distinct from all the self-folded triangles of $\tau$, with $f=(u_0,e_1,e_2,\ldots,e_n,v_0)$, $g=(v_0,e_n,\ldots,e_{2},d_1,w_0)$, $f'=(u_0',e_1',e_2',\ldots,e_n',v_0')$, $g'=(v_0',e_n',\ldots,e_{2}',d_1',w_0')$, and either
$$
\begin{cases}
\text{$e_1 = d_1$ and $\alpha=\myid_{u_0}=\beta$}, & \text{or} \\
e_1 \neq d_1.
\end{cases}
$$

If $e_1\neq d_1$, then $\kappa_0'$ either is disjoint from $f$ and from $g$, or it overlaps one of them. We analyze the situation where $\kappa_0'$ is entirely contained in $f$, and leave in the hands of the reader the analysis of the remaining possibilities.

So, suppose $\kappa_0'$ is entirely contained in $f$. Then $\kappa_0$ either is disjoint from $f'$ and from $g'$, or it overlaps one of them.
                    \begin{figure}[ht]
                \caption{}\label{Fig:kinkResConfluenceMult2Mult2SingleOverlap}
                \centering
                \includegraphics[scale=.15]{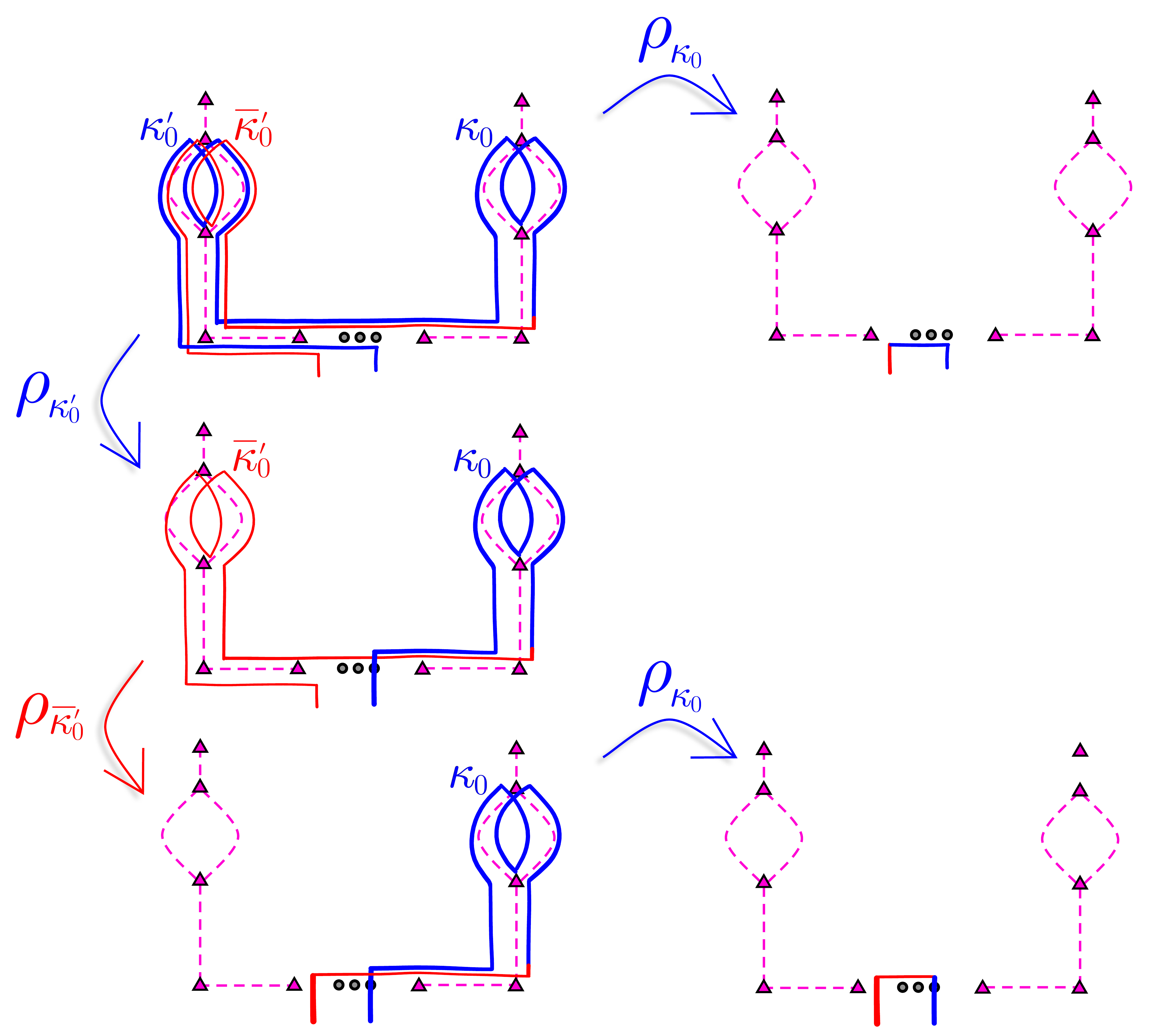}
        \end{figure}
If $\kappa_0$ is disjoint from both $f'$ and $g'$, then Figure \ref{Fig:kinkResConfluenceMult2Mult2SingleOverlap} shows us a kink $\overline{\kappa}_0'$ of $\rho_{\kappa_0'}(\varphi)$ such that $\kappa_0$ is a kink of $\rho_{\overline{\kappa}_0'}\rho_{\kappa_0'}(\varphi)$ and
$$
\rho_{\kappa_0}(\varphi)=\rho_{\kappa_0}\rho_{\overline{\kappa}_0'}\rho_{\kappa_0'}(\varphi).
$$
If $\kappa_0$ is entirely contained in $f$, then Figure \ref{Fig:kinkResConfluenceMult2Mult2DoubleOverlap}
                            \begin{figure}[ht]
                \caption{}\label{Fig:kinkResConfluenceMult2Mult2DoubleOverlap}
                \centering
                \includegraphics[scale=.15]{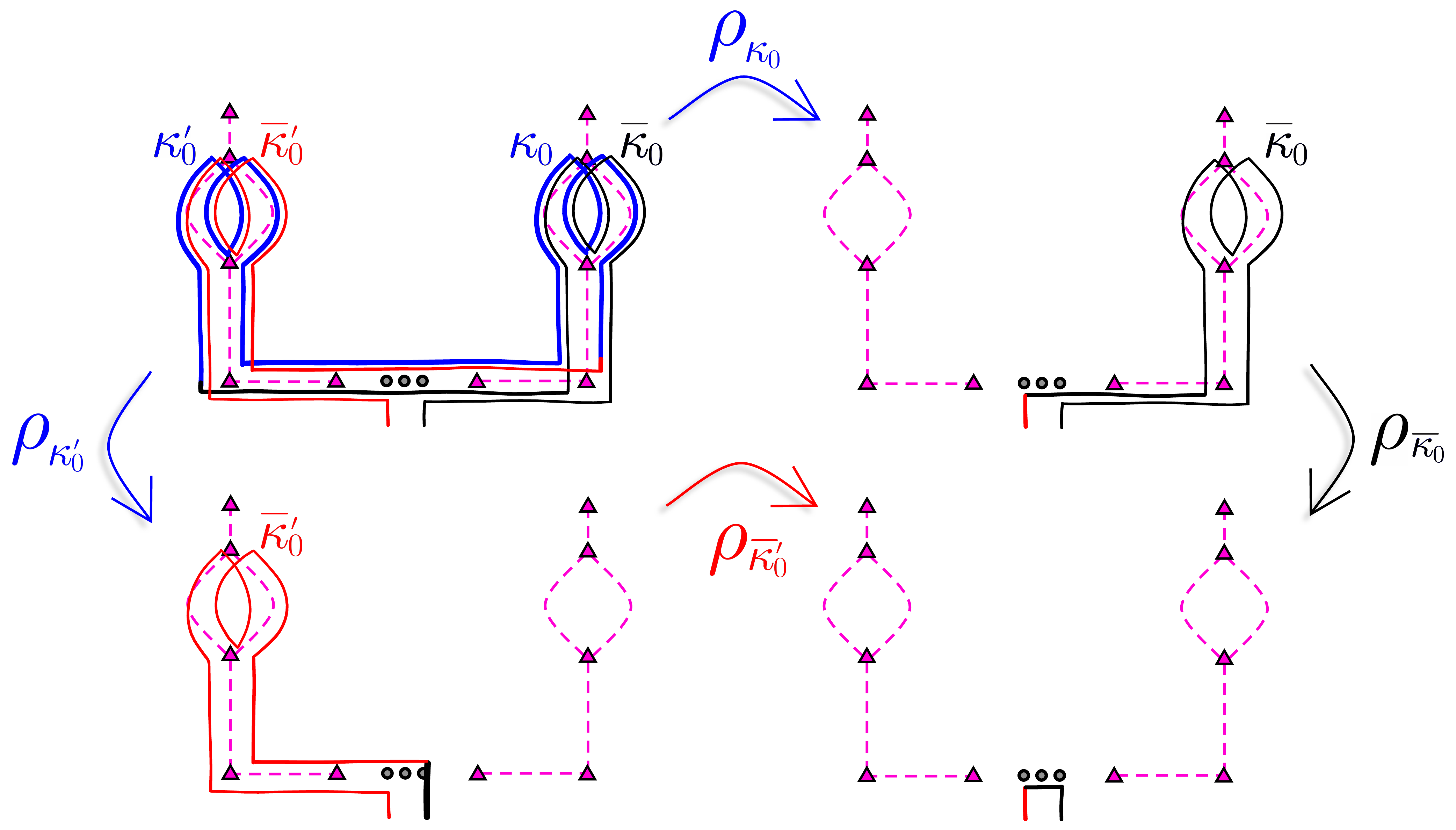}
        \end{figure}
shows us kinks $\overline{\kappa}_0$ and $\overline{\kappa}_0'$ such that
$$
\rho_{\overline{\kappa}_0'}\rho_{\kappa_0'}(\varphi)=\rho_{\overline{\kappa}_0}\rho_{\kappa_0}(\varphi).
$$
We leave in the reader's hands to analyze the cases where $\kappa_0$ is entirely contained in $g'$, or it overlaps $f'$ or $g'$ without being entirely contained in either.
\end{case}

\begin{case} If one of $\kappa_0$ and $\kappa_0'$ has multiplicity equal to $2$ and the other one has multiplicity equal to $1$, one can resort to Lemma \ref{lemma:key-lemma-for-confluence-of-kink-resolutions} as we did in Case \ref{case:mult2-mult>2} to obtain the assertion of the proposition.
\end{case}

\begin{case} If each of $\kappa_0$ and $\kappa_0'$ has multiplicity equal to $1$, then $\rho_{\kappa_0'}\rho_{\kappa_0}(\varphi)=\rho_{\kappa_0}\rho_{\kappa_0'}(\varphi)$.
\end{case}

Suppose now that $\varphi$ is a free-homotopy class as in \eqref{eq:free-homotopy-class-in-pi_1(leafyG(tau))}. If both $\kappa_0$ and $\kappa_0'$ appear as kinks of at least one of the closed backtrack-free walks $f$ representing $\varphi$, considering these walks as morphisms in $\pi_1(\leafyG(\tau))$, and if the standard forms of the morphisms
$\rho_{\kappa_0}(f)$ and $\rho_{\kappa_0'}(f)$ are closed backtrack-free walks (see Definition \ref{def:backtrack-free-walks}), then the result follows from the proof we have just given above. Otherwise, $\varphi$ is represented by one of the closed backtrack-free walks depicted in Figure \ref{Fig:interferingKinksClosedCurve} and \ref{Fig:generalAsymmetricKinkClosedCurve}
(in Figure \ref{Fig:interferingKinksClosedCurve}, $\kappa_0$ and $\kappa_0'$ are images of antipodal segments of $\mathbb{S}^1$, whereas in Figure \ref{Fig:generalAsymmetricKinkClosedCurve} they are not, see Remark \ref{rem:meaning_of_antipodal} below).
        \begin{figure}[ht]
                \caption{}\label{Fig:interferingKinksClosedCurve}
                \centering
                \includegraphics[scale=.4]{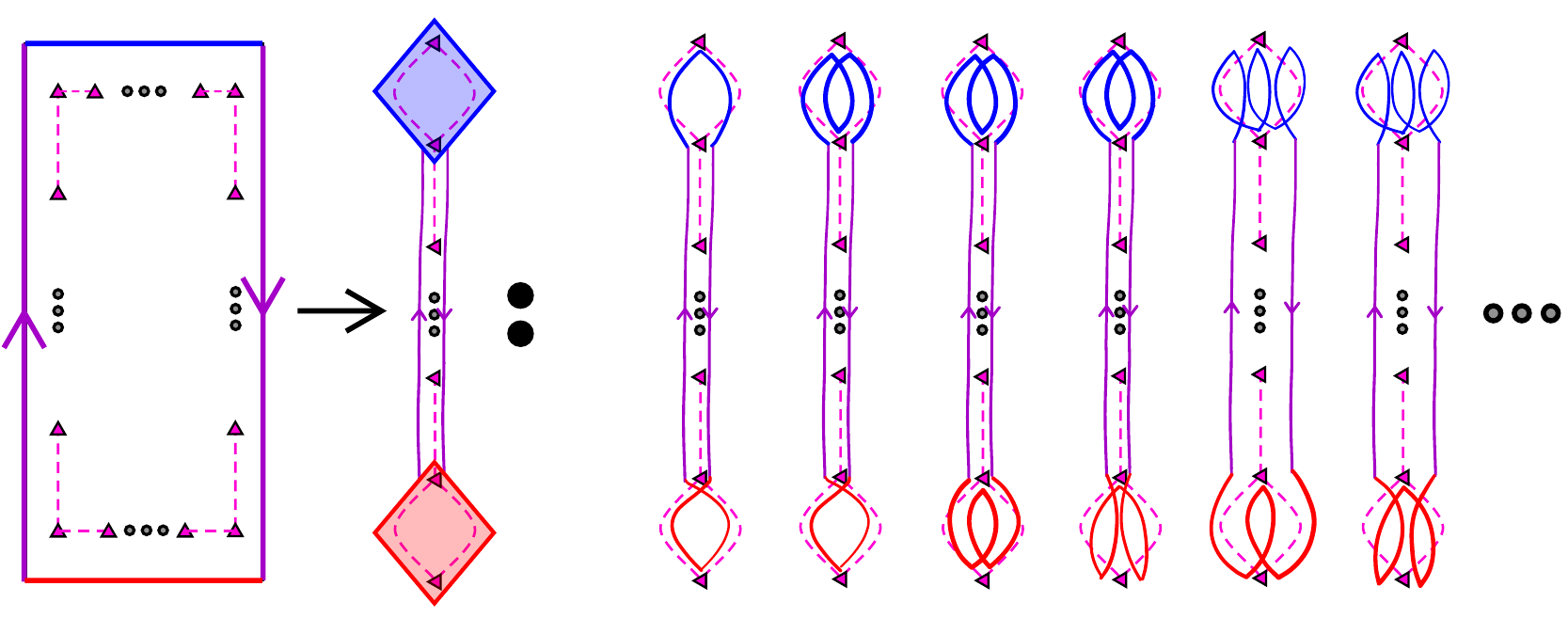}
        \end{figure}
                \begin{figure}[ht]
                \caption{}\label{Fig:generalAsymmetricKinkClosedCurve}
                \centering
                \includegraphics[scale=.4]{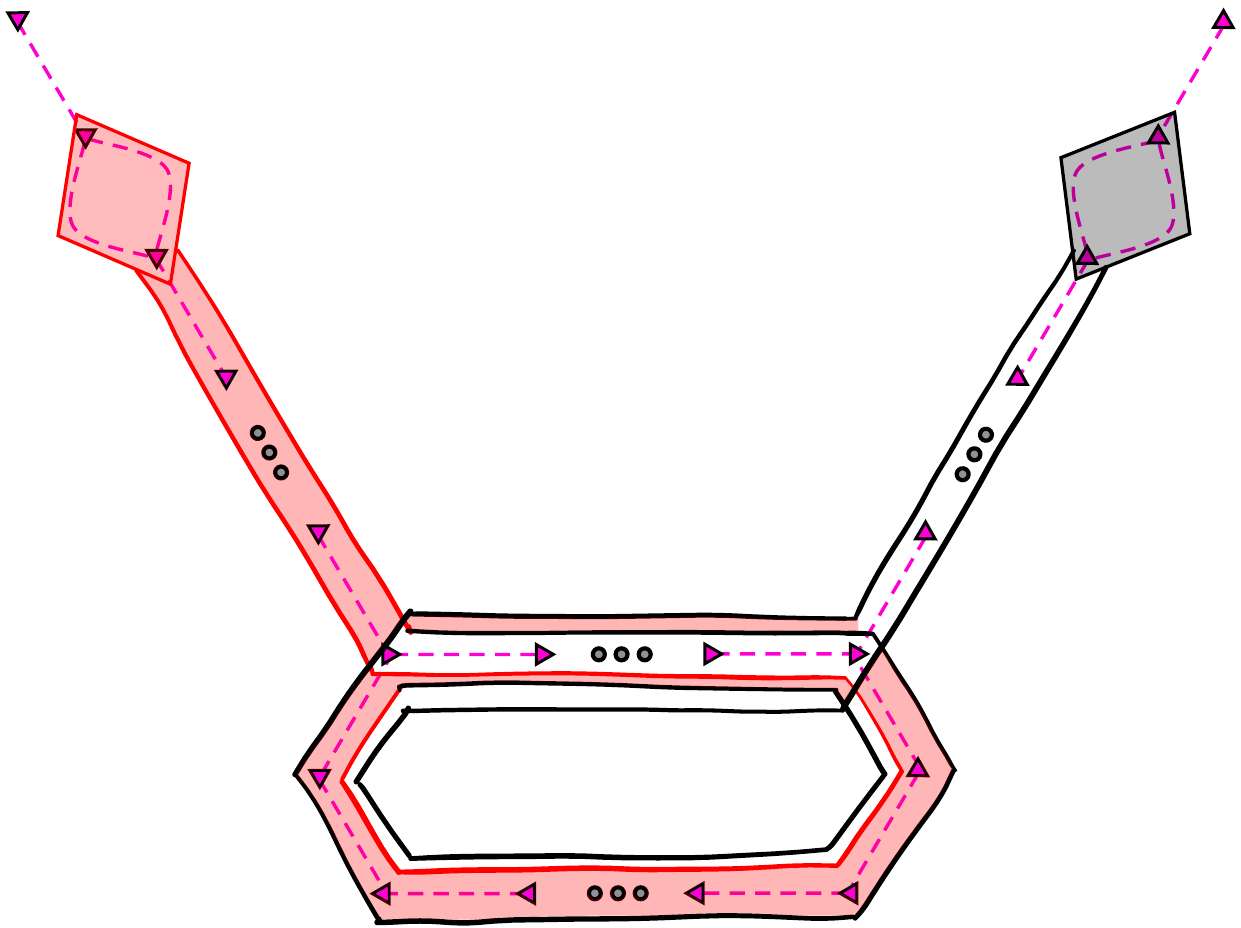}
        \end{figure}

In the situation of Figure \ref{Fig:interferingKinksClosedCurve}, the assertion of the proposition follows easily (it is here that one may have to take opposite orientations of the closed curves). The situation of Figure \ref{Fig:generalAsymmetricKinkClosedCurve} is dealt with in Figure \ref{Fig:kinkResAsymKinkClosedCurve},
                \begin{figure}
                \caption{}\label{Fig:kinkResAsymKinkClosedCurve}
                \centering
                \includegraphics[scale=.18]{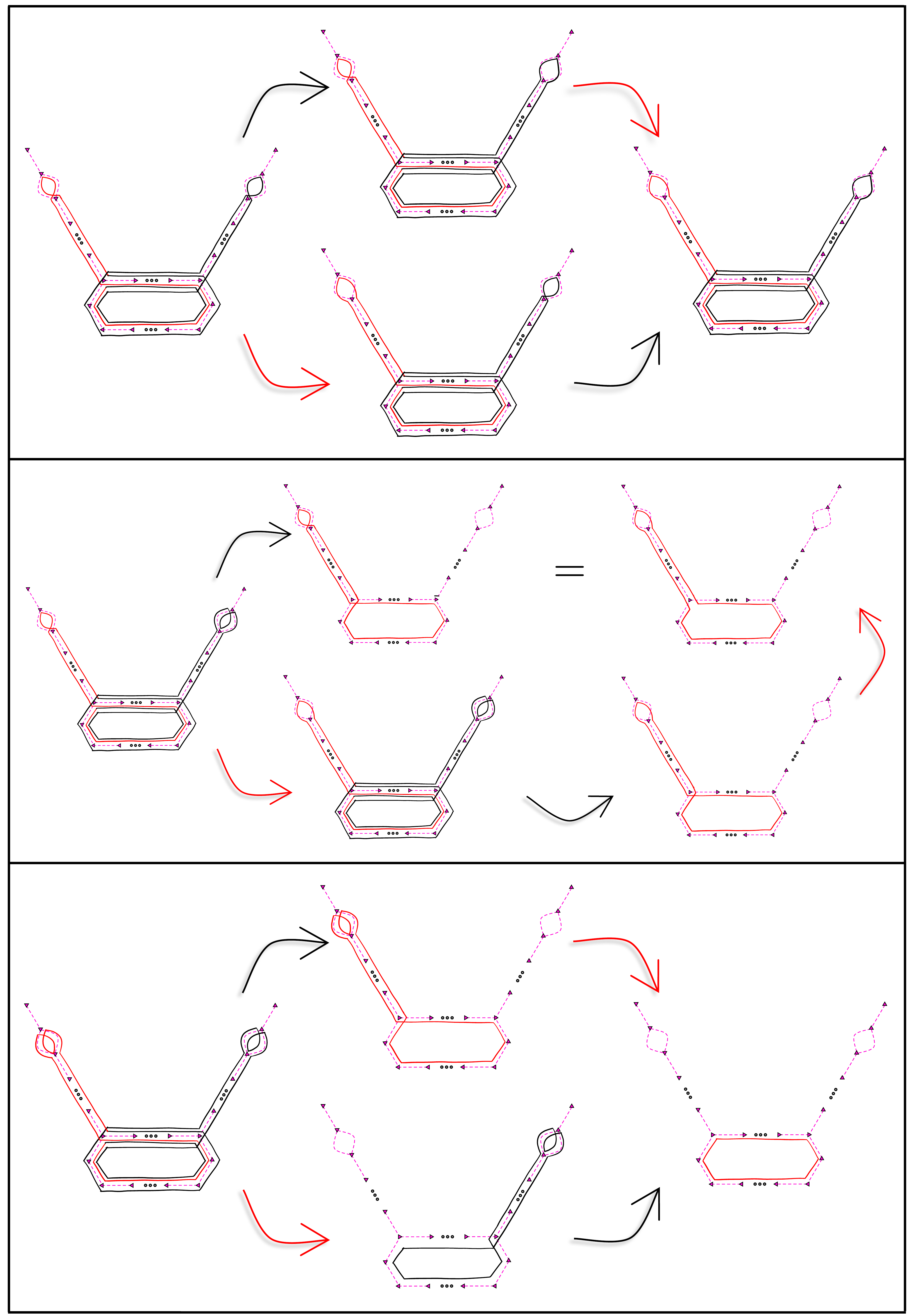}
        \end{figure}
where the assertion of the proposition is proved in the cases where the multiplicities of $\kappa_0$ and $\kappa_0'$ are less than $3$. The proposition is very easy to verify in the cases where the multiplicities are at least $3$ (see Lemma \ref{lemma:kink-resolution-makes-multiplicity-drop}).
\end{proof}

\begin{remark}\label{rem:meaning_of_antipodal}
    Consider a rectangle $Q$, made into a graph in such a way that the number of edges constituting the left vertical side of $Q$ is equal to the number of edges constituting the right vertical side of $Q$, and each of the horizontal sides of $Q$ is conformed by an even number of edges (the number of edges in the top side of $Q$ not necessarily equal to the number of edges in the bottom side). What we mean by ``$\kappa_0$ and $\kappa_0'$ are images of antipodal segments of $\mathbb{S}^1$'' in the final part of the proof of Proposition \ref{prop:local-confluence-of-kink-resolutions} is that the backtrack-free walk representing $\varphi$ is given by a graph homomorphism $Q\rightarrow \leafyG(\tau)$ as depicted in Figure \ref{Fig:interferingKinksClosedCurve}, i.e., sending the horizontal sides of $Q$ to the respective subsequences $(e_{j+1},\ldots,e_{j+2r+2})=(\eta_{v,i},\eta_{v,k},\ldots,\eta_{v_i},\eta_{v,k})$ of $\kappa_0$ and $\kappa_0'$, and the vertical sides of $Q$ to opposite orientations of the same walk on $\leafyG(\tau)$. Depending on the parities of $\kappa_0$ and $\kappa_0'$, this sometimes has the effect that resolving $\kappa_0$ first and resolving $\kappa_0'$ first results in opposite orientations of the same closed walk on $\leafyG(\tau)$.
\end{remark}

\begin{theorem}\label{thm:global-confluence-of-kink-resolutions}
    Let $\tau$ be a triangulation of signature zero, $\leafyG(\tau)$ be its leafy dual graph, and $f$ be as in in \eqref{eq:f-morfism-in-pi_1(leafyG(tau))} or \eqref{eq:free-homotopy-class-in-pi_1(leafyG(tau))}. Suppose that $(\kappa_1,\ldots,\kappa_{\ell})$ and $(\kappa_1',\ldots,\kappa_{\ell'}')$ are finite sequences of kinks such that:
    \begin{enumerate}
        \item for each $t=1,\ldots,\ell$, $\kappa_t$ is a kink of $\rho_{\kappa_{t_1}}\rho_{\kappa_{t-2}}\cdots\rho_{\kappa_0}(f)$;
        \item for each $t=1,\ldots,\ell'$, $\kappa'_t$ is a kink of $\rho_{\kappa_{t_1}'}\rho_{\kappa_{t-2}'}\cdots\rho_{\kappa_0'}(f)$;
        \item neither $\rho_{\kappa_\ell}\cdots\rho_{\kappa_1}\rho_{\kappa_0}(f)$ nor $\rho_{\kappa'_{\ell'}}\cdots\rho_{\kappa'_1}\rho_{\kappa'_0}(f)$ has kinks.
    \end{enumerate}
    Then: 
    \begin{itemize}
    \item $\rho_{\kappa_\ell}\cdots\rho_{\kappa_1}\rho_{\kappa_0}(f)=\rho_{\kappa'_{\ell'}}\cdots\rho_{\kappa'_1}\rho_{\kappa'_0}(f)$ if $\varphi$ is a morphism;
    \item $\rho_{\kappa_\ell}\cdots\rho_{\kappa_1}\rho_{\kappa_0}(\varphi)$ is equal to $\rho_{\kappa'_{\ell'}}\cdots\rho_{\kappa'_1}\rho_{\kappa'_0}(\varphi)$ or its opposite orientation if $\varphi$ is a free-homotopy class as in \eqref{eq:free-homotopy-class-in-pi_1(leafyG(tau))}.
    \end{itemize}
\end{theorem}

\begin{proof}
    By Propositions \ref{prop:contraction-after-partial-resolution-of-kink} and \ref{prop:local-confluence-of-kink-resolutions}, resolutions of kinks satisfy the hypotheses of the \emph{Diamond lemma}, see \cite[Lemma 2.4]{huet1980confluent}, \cite[Theorem 3]{newman1942on} or \cite[Theorem 1.7.10]{sapir2014combinatorial}. The result follows.
\end{proof}

\section{Kinks of curves via kinks of walks on graphs}\label{sec:kinks-of-curves}

\begin{defi}\label{def:curve-adapted-to-a-ribbon-surface}
A continuous function $\gamma:[0,1]\rightarrow\Sigma(\leafyG(\tau))$ connecting valency-$1$ vertices of $G(\tau)$ (resp. continuous closed curve $\gamma:\mathbb{S}^1\rightarrow\Sigma(\leafyG(\tau))$) is called a \emph{curve adapted to the ribbon structure of $\Sigma(\leafyG(\tau))$} if there is a finite partition $0=t_0<t_1<\ldots <t_l=1$ of the interval $[0,1]$ such that
\begin{enumerate}
    \item for $i=1,\ldots,l$, there is an edge $e_i$ of $\leafyG(\tau)$ such that the image of the open interval $(t_{i-1},t_i)$ (resp. the circle segment $e^{2\pi i (t_{i-1},t_i)}$) under $\gamma$ is entirely contained in the interior of $R_{e_i}$;
    \item for $i=2,\ldots,l$, the edges $e_{i-1}$ and $e_i$ are distinct;
    \item for $i=1,\ldots,l$, the image $\gamma([t_{i-1},t_i])\subseteq R_{e_i}$ (resp. $\gamma(e^{2\pi i[t_{i-1},t_i]})\subseteq R_{e_i}$) has exactly one element in common with each fiber of the map $\varrho_{e_i}:R_{e_i}\rightarrow e_i$ introduced in the proof of Theorem \ref{thm:strong-def-retraction-to-leafy-dual-graph};
    \item for $i,j=1,\ldots,l$, if $i\neq j$, then the intersection $\gamma([t_{i-1},t_i])\cap\gamma([t_{j-1},t_j])$ (resp. $\gamma(e^{2\pi i[t_{i-1},t_i]})\cap \gamma(e^{2\pi i[t_{j-1},t_j]})$) is a finite set.
\end{enumerate}
\end{defi}

Notice that if $\gamma:[0,1]\rightarrow\Sigma(\leafyG(\tau))$ is a curve adapted to the ribbon structure of $\Sigma(\leafyG(\tau))$, then the sequence $(\gamma(0),e_1,e_2,\ldots,e_l,\gamma(1))$ is a morphism written in standard form in the fundamental groupoid $\pi_1(\leafyG(\tau))$.

The next result can be seen to be a consequence of Theorems \ref{thm:morphisms-represented-by-walks-on-graphs} and \ref{thm:strong-def-retraction-to-leafy-dual-graph}. Alternatively, it follows from \cite{hass1985intersections}.

\begin{theorem}\label{thm:every-curve-is-homotopic-to-an-adapted-curve} Let $\tau$ be a signature-zero triangulation, and let $\varrho:\Sigma(\leafyG(\tau))\rightarrow\leafyG(\tau)$ be the strong deformation retraction defined in the proof of Theorem \ref{thm:strong-def-retraction-to-leafy-dual-graph}.
\begin{enumerate}
\item Every continuous curve $\delta:[0,1]\rightarrow\Sigma(\leafyG(\tau))$ connecting valency-$1$ vertices of $G(\tau)$ is homotopic $\operatorname{rel}$ $\{0,1\}$ to a curve $\gamma:[0,1]\rightarrow\Sigma(\leafyG(\tau))$ adapted to the ribbon structure of $\Sigma(\leafyG(\tau))$, say, with respect to a partition $0=t_0<t_1<\ldots <t_l=1$ of the interval $[0,1]$. Furthermore, the sequence $(\gamma(0),e_1,e_2,\ldots,e_l,\gamma(1))$ from Definition \ref{def:curve-adapted-to-a-ribbon-surface} is a backtrack-free walk that coincides with the standard form of $\varrho\delta$ and is, thus, uniquely determined by $\delta$.
    \item Every continuous closed curve $\delta:\mathbb{S}^1\rightarrow\Sigma(\leafyG(\tau))$ is freely homotopic to a closed curve $\gamma:\mathbb{S}^1\rightarrow\Sigma(\leafyG(\tau))$ adapted to the ribbon structure of $\Sigma(\leafyG(\tau))$, say, with respect to a partition $0=t_0<t_1<\ldots <t_l=1$ of the interval $[0,1]$. 
    Furthermore, the sequence $(\gamma(1),e_1,e_2,\ldots,e_l,\gamma(1))$ from Definition \ref{def:curve-adapted-to-a-ribbon-surface}  is a closed backtrack-free walk whose rotational equivalence class
    coincides with the standard form of $\varrho\delta$ and is, thus, uniquely determined by $\delta$.
\end{enumerate}
\end{theorem}

\begin{defi}\label{def:kink-of-curve-wrt-triangulation-of-signature-zero} Let $\tau$ be a signature-zero triangulation, and let $\varrho:\Sigma(\leafyG(\tau))\rightarrow\leafyG(\tau)$ be the strong deformation retraction defined in the proof of Theorem \ref{thm:strong-def-retraction-to-leafy-dual-graph}.
Let 
$\gamma:[0,1]\rightarrow\Sigma(\leafyG(\tau))$ (resp. $\gamma:\mathbb{S}^1\rightarrow\Sigma(\leafyG(\tau))$) be a curve adapted to the ribbon structure of $\Sigma(\leafyG(\tau))$ with respect to a partition $0=t_0<t_1<\ldots <t_l=1$ of the interval $[0,1]$. A \emph{kink of $\gamma$ with respect to $\tau$} is 
a segment $\gamma|_{[t_j,t_m]}:[t_j,t_m]\rightarrow\Sigma(\leafyG(\tau))$ (resp. $\gamma|_{e^{2\pi i [a,b]}}:e^{2\pi i [a,b]}\rightarrow\Sigma(\leafyG(\tau))$) such that the subsequence $(e_j,\ldots,e_m)$ of $\phi(\varrho\delta)$ it determines is a kink of $\phi(\varrho\delta)$.
\end{defi}

\begin{proposition}\label{prop:having-kinks-is-independent-of-triangulation}
    Suppose that $\tau$ and $\sigma$ are triangulations of signature zero. Let $\gamma_1$ and $\gamma_2$ be curves adapted to the ribbon structures of $\Sigma(\leafyG(\tau))$ and $\Sigma(\leafyG(\sigma))$, respectively. If $\iota_\tau(\gamma_1)$ and $\iota_\sigma(\gamma_2)$ are homotopic $\operatorname{rel}$ $\{0,1\}$ (resp. freely homotopic) in $\Sigma$ (see \eqref{eq:embedding-of-ribbon-surfaces}), then $\gamma_1$ has at least one kink if and only if $\gamma_2$ has at least one kink.
\end{proposition}

\begin{proof} Any two triangulations of $\surf$ of signature zero are related by a finite sequence of $\lozenge$-flips (see \cite[Definition 9.11]{fomin2008cluster}). Hence it suffices to prove the proposition in the case where $\tau$ and $\sigma$ are related by a single $\lozenge$-flip. Furthermore, it is enough to show that, \emph{given} $\gamma_1$, \emph{there exists} a $\gamma_2$ with the properties that $\iota_\tau(\gamma_1)$ and $\iota_\sigma(\gamma_2)$ are homotopic $\operatorname{rel}$ $\{0,1\}$ (resp. freely homotopic) in $\Sigma$, and $\gamma_1$ has at least one kink if and only if $\gamma_2$ has at least one kink.

There are two types of $\lozenge$-flips, namely, the ones sketched in Figures \ref{Fig:standardFlip} and \ref{Fig:doubleFlip}.
        \begin{figure}[ht]
                \caption{}\label{Fig:standardFlip}
                \centering
                \includegraphics[scale=.14]{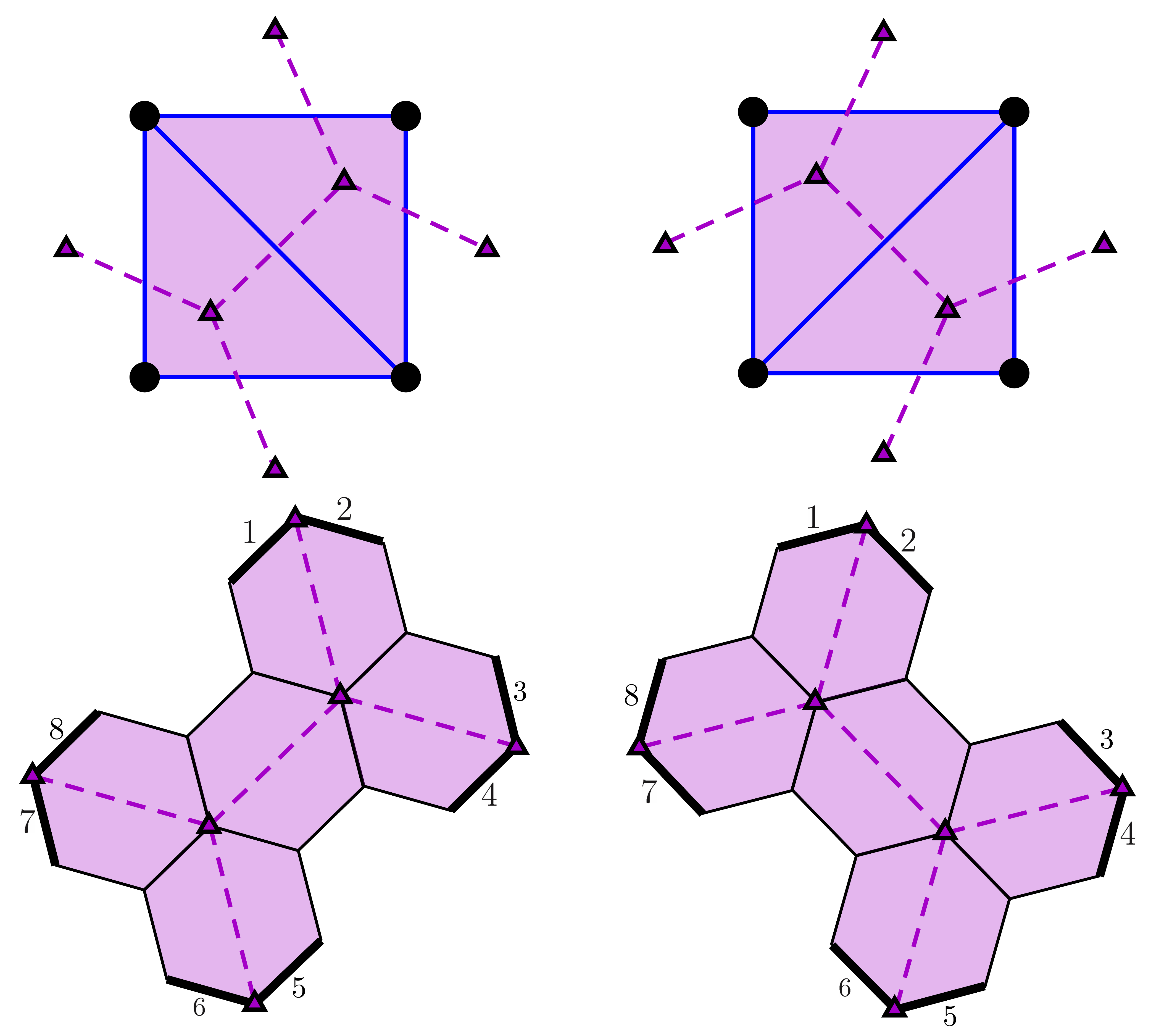}
        \end{figure}
                \begin{figure}[ht]
                \caption{}\label{Fig:doubleFlip}
                \centering
                \includegraphics[scale=.4]{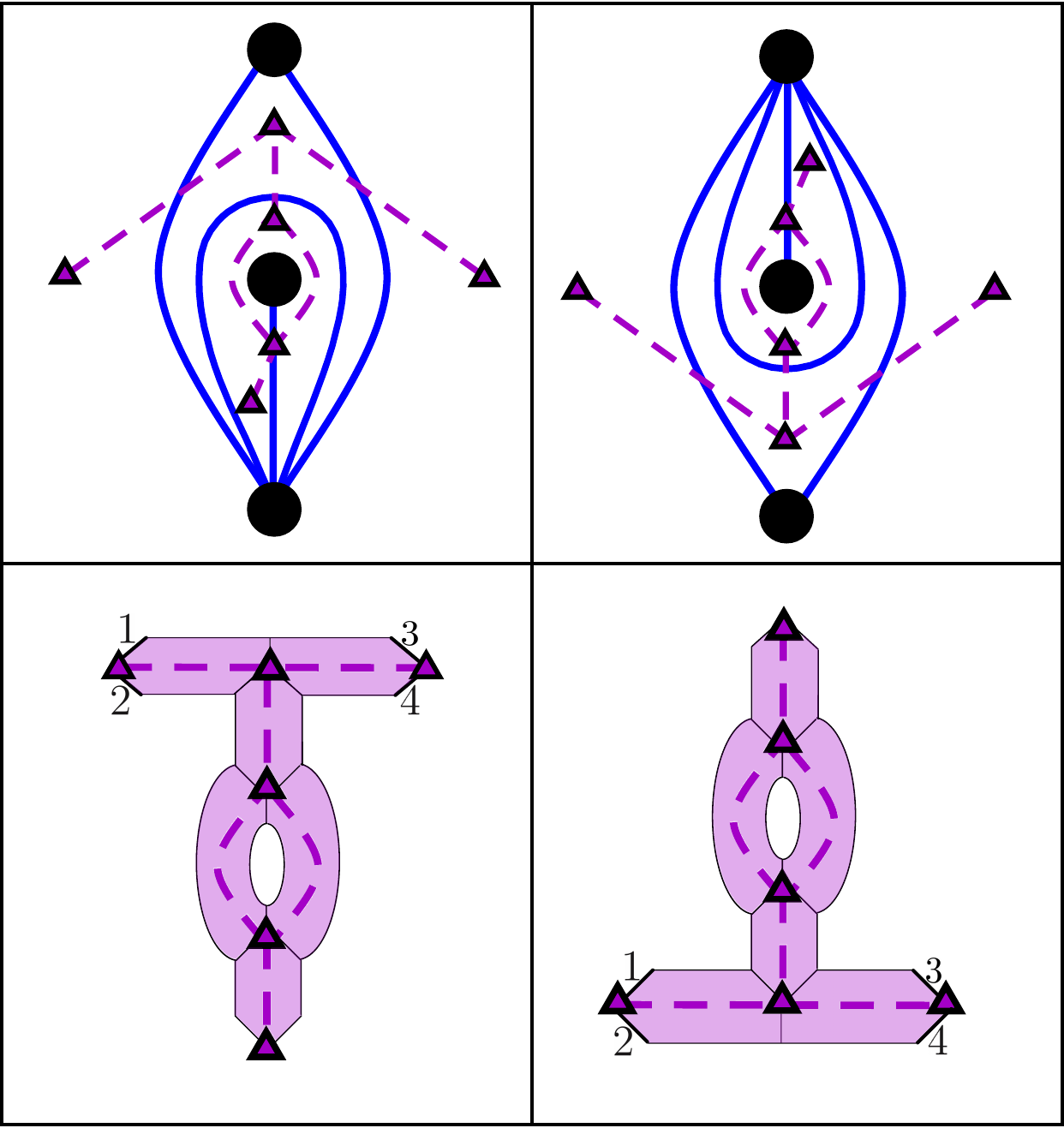}
        \end{figure}
In each of these figures, we have drawn the subgraphs of $\leafyG(\tau)$ and $\leafyG(\sigma)$ conformed by the edges incident to the triangles of $\tau$ and $\sigma$ that contain the arcs involved in the $\lozenge$-flip. We have also drawn the ribbon subsurfaces of $\Sigma(\leafyG(\tau))$ and $\Sigma(\leafyG(\sigma))$ induced by these subgraphs. We shall denote these ribbon subsurfaces by $S(\tau)$ and $S(\sigma)$.

Outside the portions depicted in Figures \ref{Fig:standardFlip} and \ref{Fig:doubleFlip}, $\leafyG(\tau)$ and $\leafyG(\sigma)$ are identical, and $\Sigma(\leafyG(\tau))$ and $\Sigma(\leafyG(\sigma))$ are identical. Thus, in order to exhibit a $\gamma_2$ if we are given $\gamma_1$, it suffices to replace each segment $s$ of $\gamma_1$ contained in $S(\tau)$ by a segment $s'$ contained in $S(\sigma)$, connecting the endpoints of $s$, and satisfying the conditions in Definition \ref{def:curve-adapted-to-a-ribbon-surface}, in such a way that $\iota_\tau(s)$ and $\iota_\sigma(s')$ are homotopic rel $\{0,1\}$ in $\Sigma\setminus(\punct\cup\partial\Sigma)$. This replacement can be carried out combinatorially very easily, and has been sketched in Figure \ref{Fig:FlipHomeo}.
                \begin{figure}[ht]
                \caption{}\label{Fig:FlipHomeo}
                \centering
                \includegraphics[scale=.16]{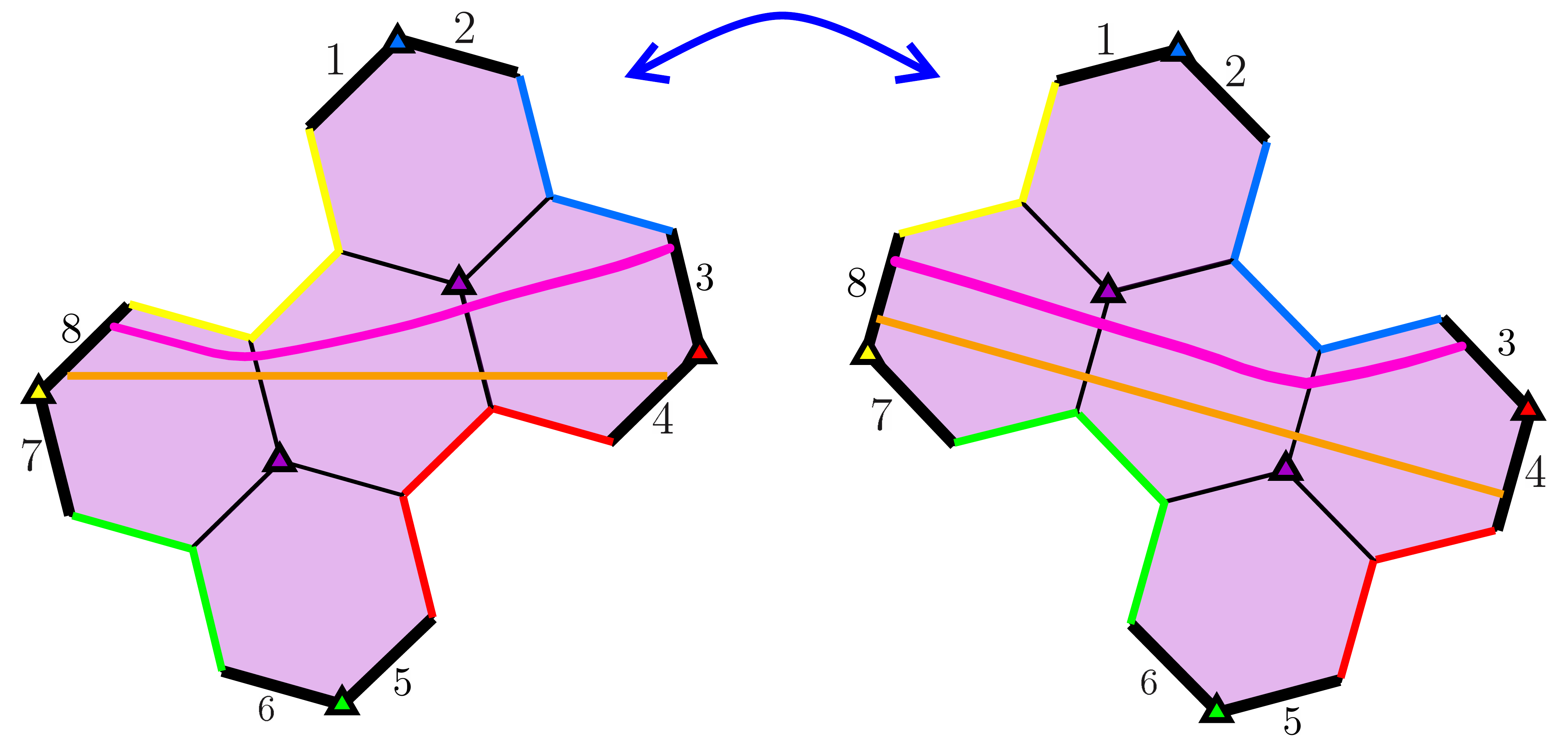} \\
                \includegraphics[scale=.31]{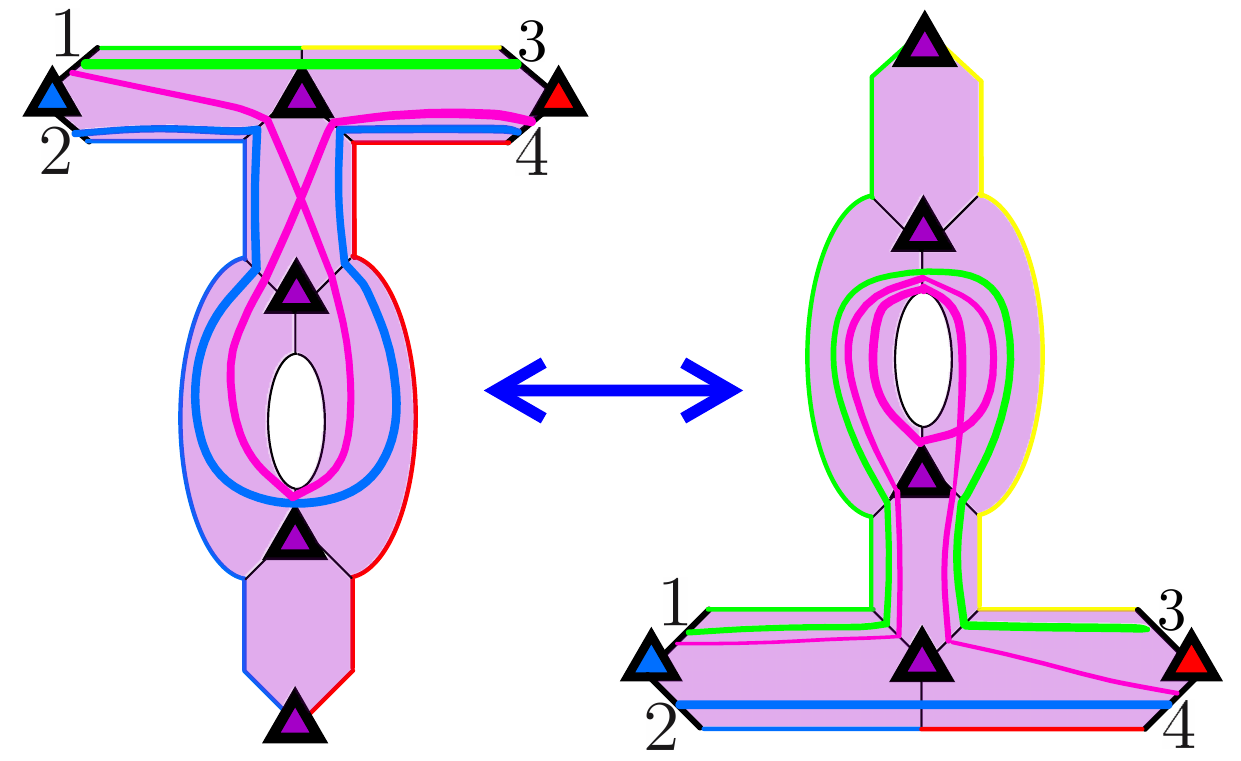}
        \end{figure}

Letting $\gamma_2$ be the result of performing all of the above replacements, it becomes routine to verify that $\gamma_1$ has a kink with respect to $\tau$ if and only if $\gamma_2$ has a kink with respect to~$\sigma$.
\end{proof}

\begin{remark}
    With the current definition of multiplicity of a kink on a (closed) backtrack-free walk on $\leafyG(\tau)$, see Definition \ref{def:what-is-a-kink-on-a-graph}, for kinks of positive multiplicity it is not always true that their multiplicity is invariant under $\lozenge$-flips. However, a suitable refinement of the notion of multiplicity can be given that is invariant under $\lozenge$-flips. It would also be interesting to define the notion of kink with respect to a triangulation not necessarily having signature zero, and proving the corresponding invariance under flips.
\end{remark}

\begin{remark} By the work of Hansper \cite{hansper2022classification}, the two possible notions of \emph{admissible} string or band on a clannish algebra are equivalent, see also \cite{geiss2023onhomomorphism}. Now, it is extremely easy to see that if $\tau$ is a triangulation of signature zero, then the Jacobian algebra $\Lambda(\tau)$, associated to $\tau$ by the second author in \cite{labardini2009quivers}, is a clannish algebra. It turns out that the string or band on $\Lambda(\tau)$ determined by a curve on $\surf$ is admissible if and only if the curve itself has no kinks. 
\end{remark}

\section{Equivalence of orbifold fundamental groupoids}\label{sec:equiv-of-fund-groupoids}

Let $\surf$ be a surface with non-empty boundary. We define a subgroupoid of the fundamental groupoid $\pi_1(\Sigma\setminus\punct)$ as follows. Form a set $E$ by picking exactly one point from each connected component of $(\partial\Sigma)\setminus\marked$.  Define $\pi_1(\Sigma\setminus\punct,E)$ to be the full subcategory of $\pi_1(\Sigma\setminus\punct)$ having $E$ as set of objects.

For each point $u\in E$, each puncture $p\in \punct$, and each continuous function $c:[0,1]\rightarrow \Sigma$ connecting $u$ to $p$ and such that $(\partial\Sigma\cup\punct)\cap c(0,1)=\varnothing$, take a continuous function $\gamma_{u,p,c}:[0,1]\rightarrow \Sigma$ with $\gamma_{u,p,c}(0)=u=\gamma_{u,p,c}(1)$, $(\partial\Sigma\cup\punct)\cap c(0,1)=\varnothing$, closely following $c$ but surrounding $p$ clockwisely instead of going into it. Thus, $\gamma_{u,p,c}$ cuts out a once-punctured monogon based at $u$, the puncture being $p$.

For each $u\in E$ form the set
$$
S(u):=\{[\gamma_{u,p,c}]^2\suchthat p\in\punct \ \text{and} \ c \ \text{is as above}\}\subseteq \pi_1(\Sigma\setminus\punct,E)(u,u)= \pi_1(\Sigma\setminus\punct)(u,u).
$$
Although $S(u)$ fails to be a subgroup of $\pi_1(\Sigma\setminus\punct,E)(u,u)$, 
the collection $S:=(S(u))_{u\in E}$ has the property that for every $u_0,v_0\in E$, every $h\in S(u_0)$ and every morphism
$g\in\pi_1(\Sigma\setminus\punct,E)(x,y)$, we have $ghg^{-1}\in S(v_0)$. Thus, if we let $H(u)$ be the subgroup of $\pi_1(\Sigma\setminus\punct,E)(u,u)$ generated by $S(u)$, then $H:=(H(u))_{u\in E}$ is a normal multilocular subgroup of $\pi_1(\Sigma\setminus\punct,E)$.

\begin{defi}\cite{chas2016the,amiot2021the}
    The quotient groupoid
    $$
    \pi_{1,\punct}^{\operatorname{orb}}(\Sigma,E):=\pi_1(\Sigma\setminus\punct,E)/H
    $$
    will be called the \emph{$2$-orbifold fundamental groupoid} of $\surf$.
\end{defi}

Let $\tau$ be a triangulation of $\surf$ having signature zero.
Each point $u$ in $E$ belongs to exactly one boundary segment $b_u$ of $\Sigma$. This segment $b_u$ is a vertex of both of the graphs $G(\tau)$ and $\leafyG(\tau)$. We will identify each $u$ as the vertex $b_u$ whenever we consider either of these graphs. Under this identification, we define $\pi_1(\leafyG(\tau),E)$ to be the full subcategory of $\pi_1(\leafyG(\tau))$ having $E$ as set of objects.

For each vertex $u\in E$, each self-folded triangle $v$ of $\tau$ and each morphism $c\in \leafyG(\tau)(u,v)$, set $\gamma_{u,v,c}:=c*(\eta_{v,1},\eta_{v_2})*c^{-1}=c^{-1}\circ (\eta_{v,1},\eta_{v_2})\circ c$. For each $u\in E$ form the set
$$
R(u):=\{\gamma_{u,v,c}^2\suchthat v \ \text{is as above} \ \text{and} \ c\in \leafyG(\tau)(u,v)\}\subseteq \pi_1(\leafyG(\tau),E)(u,u)=\pi_1(\leafyG(\tau))(u,u).
$$
Although $R(u)$ fails to be a subgroup of $\pi_1(\leafyG(\tau),E)(u,u)$, 
the collection $R:=(R(u))_{u\in E}$ has the property that for every $u_0,v_0\in E$, every $h\in R(u_0)$ and every morphism
$g\in\pi_1(\leafyG(\tau),E)(x,y)$, we have $ghg^{-1}\in R(v_0)$. Thus, if we let $K(u)$ be the subgroup of $\pi_1(\Sigma\setminus\punct,E)(u,u)$ generated by $R(u)$, then $K:=(K(u))_{u\in E}$ is a normal multilocular subgroup of $\pi_1(\leafyG(\tau),E)$.

\begin{defi}\label{def:orb-fund-groupoid-of-leafyG}
    The quotient groupoid
    $$
    \pi_{1,\punct}^{\operatorname{orb}}(\leafyG(\tau),E):=\pi_1(\leafyG(\tau),E)/K
    $$
    will be called the \emph{$2$-orbifold fundamental groupoid} of the graph $\leafyG(\tau)$.
\end{defi}

The proof of the next result follows from the definitions.

\begin{theorem}\label{thm:comm-diagram-orb-fund-groupoids}
Let $\surf$ be a surface with non-empty boundary, and
let $\tau$ be an ideal triangulation of $\surf$ of signature zero. The strong deformation retraction $\rho:\Sigma\setminus\punct\rightarrow \leafyG(\tau)$ from Theorem \ref{thm:strong-def-retraction-to-leafy-dual-graph} induces a commutative diagram of groupoids
$$
\xymatrix{\pi_1(\Sigma\setminus\punct,E) \ar[d]_{\mathfrak{p}} \ar[r]^{\rho_{\#}}_{\cong} & \pi_1(\leafyG(\tau),E) \ar[d]^{\mathfrak{p}} \\
\pi_{1,\punct}^{\operatorname{orb}}(\Sigma,E) \ar[r]_{\overline{\rho_{\#}}}^{\cong} & \pi_{1,\punct}^{\operatorname{orb}}(\leafyG(\tau),E)
}
$$
whose horizotal arrows are isomorphisms.
\end{theorem}

\section{Main result: uniqueness of kink-free representative curves}\label{sec:main-result}

Let $\surf$ and $E$ be as in Section \ref{sec:equiv-of-fund-groupoids}. For each $u\in E$, define $L(u)$ to be the image, under the projection $\pi_1(\Sigma\setminus\mathbb{P})(u,u)\rightarrow \pi_{1,\punct}^{\operatorname{orb}}(\Sigma,E)(u,u)$,
of the set of homotopy classes of the curves $\gamma_{u,p,c}$ described in the second paragraph of Section \ref{sec:equiv-of-fund-groupoids}. Since $\pi_1(\Sigma\setminus\mathbb{P})(u,u)$ is isomorphic to a torsion-free Fuchsian group having a finite Dirichlet polygon of finite hyperbolic area, we can deduce from e.g. \cite[Theorem 3.5.2, Corollary 4.2.6]{katok1992fuchsian} and \cite[Proposition VI.1.4]{deSaintGervais2016uniformization} that $L(u)=\{f\in \pi_{1,\punct}^{\operatorname{orb}}(\Sigma,E)(u,u)\suchthat f\neq \myid$ and has finite order$\}$, and thus, $L(u)=\{f\in \pi_{1,\punct}^{\operatorname{orb}}(\Sigma,E)(u,u)\suchthat \myid\neq f=f^{-1}\}$.

\begin{theorem}\label{thm:unique-kink-free-representatives}
    Let $\surf$ be a surface with non-empty boundary. 
    \begin{itemize}
    \item Given $u_0,v_0\in E$, there exists exactly one function $$\iota:\pi_{1,\punct}^{\operatorname{orb}}(\Sigma,E)(u_0,v_0)\setminus L(u_0) \rightarrow\pi_1(\Sigma\setminus\punct,E)(u_0,v_0)$$ with the following properties:
    \begin{enumerate}
        \item $\mathfrak{p}\circ\iota=\myid$;
        \item for every $f\in \pi_{1,\punct}^{\operatorname{orb}}(\Sigma,E)(u_0,v_0)\setminus L(u_0)$ there exists a representative curve $\gamma\in \iota(f)$ that has no kinks.
    \end{enumerate}
    \item
    there is exactly one function $$\iota:\pi_{1,\punct}^{\operatorname{orb},\operatorname{free}}(\Sigma,E)/\sim \rightarrow\pi_1^{\operatorname{free}}(\Sigma\setminus\punct,E)/\sim$$ with the following properties:
    \begin{enumerate}
        \item $\mathfrak{p}\circ\iota=\myid$;
        \item for every $f\in \pi_{1,\punct}^{\operatorname{orb},\operatorname{free}}(\Sigma,E)/\sim$ there exists a representative curve $\gamma\in \iota(f)$ that has no kinks;
    \end{enumerate}
    where $\sim$ are the equivalence relations that identify each closed curve with its opposite orientation.
    \end{itemize}
\end{theorem}

\begin{proof} We prove the first statement, the proof of the second statement is similar.
By Proposition \ref{prop:having-kinks-is-independent-of-triangulation} and Theorem \ref{thm:comm-diagram-orb-fund-groupoids}, it is enough to take a triangulation $\tau$ of signature zero, and prove the statements for $\pi_{1,\punct}^{\operatorname{orb}}(\leafyG(\tau),E)$ and $\pi_1(\leafyG(\tau),E)$ instead of $\pi_{1,\punct}^{\operatorname{orb}}(\Sigma,E)$ and $\pi_1(\Sigma\setminus\punct,E)$.

Take $f\in \pi_1(\leafyG(\tau),E)(u_0,v_0)\setminus \overline{\rho_{\#}}(L(u_0))$, where $\overline{\rho_{\#}}$ is the morphism of groupoids in Theorem \ref{thm:comm-diagram-orb-fund-groupoids}. Assume $f$ is written in standard form. It is clear that applying resolutions of kinks to $f$ does not affect the class of $f$ in $\pi_{1,\punct}^{\operatorname{orb}}(\leafyG(\tau),E)(u_0,v_0)$. Furthermore, by Theorem \ref{thm:global-confluence-of-kink-resolutions}, the element $\iota(f)\in\pi_1(\leafyG(\tau),E)(u_0,v_0)$ obtained after fully resolving all kinks from $f$ is independent of the order in which the kinks are resolved. This produces one function $$\iota:\pi_{1,\punct}^{\operatorname{orb}}(\Sigma,E)(u_0,v_0)\setminus L(u_0)\rightarrow\pi_1(\Sigma\setminus\punct,E)(u_0,v_0)$$
with the two stated properties.

Uniqueness follows also from Theorem \ref{thm:global-confluence-of-kink-resolutions} after realizing that, by Definition \ref{def:orb-fund-groupoid-of-leafyG}, the paragraph preceding it and Lemma \ref{lemma:key-lemma-for-confluence-of-kink-resolutions}, if $g\in\pi_1(\leafyG(\tau),E)(u_0,v_0)$ is a morphism without kinks and such that $\mathfrak{p}(g)=\mathfrak{p}(f)$, then it is possible to obtain $g$ from $f$ by applying a finite sequence of resolutions of kinks.
\end{proof}

\begin{remark}
The functions $\iota$ in the first assertion of Theorem \ref{thm:unique-kink-free-representatives} may fail to constitute a morphism of groupoids, see Figure \ref{Fig:iota_not_groupoid_morphism}.
\begin{figure}[ht]
                \caption{}\label{Fig:iota_not_groupoid_morphism}
                \centering
                \includegraphics[scale=.3]{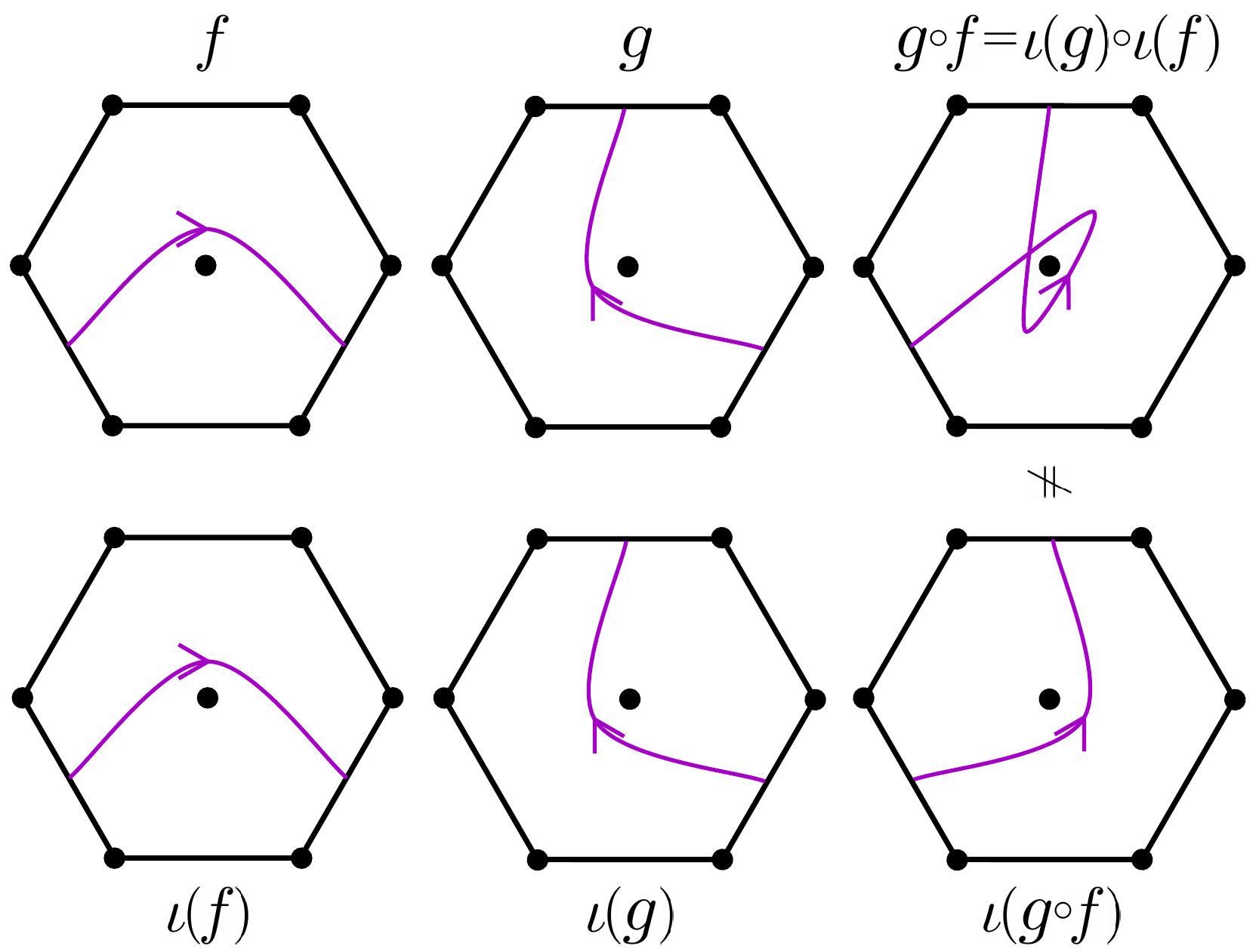}
        \end{figure}
\end{remark}

%
%
%
%

\bibliographystyle{amsalpha-fi-arxlast}
\bibliography{AGT_GeissLabardini_GroupoidsOfGraphs.bib}

\end{document}